\documentclass[12pt]{article}
\usepackage[latin1]{inputenc}
\usepackage{amsmath,amsthm,amssymb}
\usepackage{amsfonts}
\usepackage{amsmath,amsthm,amssymb,amscd}
\usepackage{latexsym}
\usepackage{color}
\usepackage{graphicx}
\usepackage{mathrsfs}

\usepackage{color,enumitem,graphicx}
\usepackage[colorlinks=true,urlcolor=black,
citecolor=black,linkcolor=black,linktocpage,pdfpagelabels,
bookmarksnumbered,bookmarksopen]{hyperref}

\makeatletter \@addtoreset{equation}{section} \makeatother

\newtheorem{theorem}{Theorem}[section]

\newtheorem{proposition}{Proposition}[section]
\newtheorem{lemma}{Lemma}[section]
\newtheorem{remark}{Remark}[section]

\allowdisplaybreaks

\begin{document}
\title{Boundary concentration of peak solutions  for fractional Schr\"{o}dinger-Poisson system}

\author{Shengbing Deng\footnote{E-mail address:\, {\tt shbdeng@swu.edu.cn} (S. Deng), {\tt xltian@email.swu.edu.cn} (X. Tian)}\ \footnote{The research has been supported by National Natural Science Foundation of China (No. 11971392).}\ \ and Xingliang Tian\\
\footnotesize  School of Mathematics and Statistics, Southwest University,
Chongqing, 400715, P.R. China}

\date{ }
\maketitle

\begin{abstract}
{The goal of this paper is to study the existence of peak solutions for the following fractional Schr\"{o}dinger-Poisson system:
 \begin{eqnarray*}
    \left\{ \arraycolsep=1.5pt
       \begin{array}{ll}
        \varepsilon^{2s}(-\Delta)^{s}u+u+\phi u=u^p,\ \ \ &\ \mbox{in}\ \Omega,\\[2mm]
        (-\Delta)^{s}\phi=u^2,\ \ \ &\ \mbox{in}\ \Omega,\\[2mm]
        u=\phi=0,\ \ \ \ &\ \mbox{in}\ \mathbb{R}^N\setminus \Omega,
        \end{array}
    \right.
    \end{eqnarray*}
    where $s\in(0,1)$, $N>2s$, $p\in (1,\frac{N+2s}{N-2s})$, $\Omega$ is a bounded domain in $\mathbb{R}^N$ with Lipschitz boundary, and $(-\Delta)^{s}$ is the fractional Laplacian operator, $\varepsilon$ is a small positive parameter. By using the Lyapunov-Schmidt reduction method, we construct a single peak solution $(u_\varepsilon,\phi_\varepsilon)$ such that the peak of $u_\varepsilon$ is in the domain but near the boundary. In order to characterize the boundary concentration of solutions,  which concentrates at an approximate distance $\varepsilon^{2/3}$ away from the boundary $\partial\Omega$ as $\varepsilon$ tends to 0, some new
    estimates and analytic technique are used.}



\smallskip
\emph{\bf Data availability statement:} Data sharing not applicable to this article as no datasets were generated or analysed during the current study.
\end{abstract}

\section{{\bfseries Introduction}}\label{sectir}

    Let $\Omega$ be a bounded domain in $\mathbb{R}^N$ with Lipschitz boundary, we consider  the following fractional Schr\"{o}dinger-Poisson system
    \begin{eqnarray}\label{P}
    \left\{ \arraycolsep=1.5pt
       \begin{array}{ll}
        \varepsilon^{2s}(-\Delta)^{s}u+u+\phi u=u^p,\ \ \ &\ \mbox{in}\ \Omega,\\[2mm]
        (-\Delta)^{s}\phi=u^2,\ \ \ &\ \mbox{in}\ \Omega,\\[2mm]
        u=\phi=0,\ \ \ \ &\ \mbox{in}\ \mathbb{R}^N\setminus \Omega,
        \end{array}
    \right.
    \end{eqnarray}
    where $s\in(0,1)$,  $p\in (1,\frac{N+2s}{N-2s})$,  $N>2s$, $\varepsilon$ is a small positive parameter, 
    and $(-\Delta)^{s}$ is the fractional Laplacian operator defined by
    \begin{equation*}
    (-\Delta)^{s}w(x):=C_{N,s}\ {\rm P.V.} \int_{\mathbb{R}^N}\frac{w(x)-w(y)}{|x-y|^{N+2s}}dy=2\lim_{\epsilon\rightarrow 0^+}\int_{\mathbb{R}^N\backslash B_{\epsilon}(x)}\frac{w(x)-w(y)}{|x-y|^{N+2s}}dy,
    \end{equation*}
    where $C_{N,s}=2^{2s-1}\pi^{-\frac{N}{2}}\frac{\Gamma(\frac{N+2s}{2})}{\Gamma(-s)}$. From the definition of $(-\Delta)^s$, we can know that the fractional Laplacian problems are nonlocal. The fractional spaces and the corresponding nonlocal equations have many important applications in various fields of science and engineering, we refer to \cite{dpv} and references therein.

    This paper deals with the {\em semiclassical limit} of system (\ref{P}), i.e. it is concerned with the problem of finding nontrivial solutions and studying their asymptotic behavior when $\varepsilon\to 0^+$; hence such solutions are usually referred to as {\em semiclassical} ones. The analysis of the Schr\"{o}dinger-Poisson equations in the limit $\varepsilon\to 0^+$ is not only a challenging mathematical task, but also of some relevance for understanding of a wide class of quantum phenomena. Indeed, according to the {\em correspondence principle}, letting $\varepsilon$ tends to zero in the Schr\"{o}dinger equation formally describes the transition from quantum mechanics to classical mechanics; it is therefore interesting to study which kind of semiclassical phenomena system (\ref{P}) exhibits.

    When $s=1$ and $\Omega=\mathbb{R}^3$, (\ref{P}) has been investigated by \cite{rv11} and also \cite{lz17} for $\Omega=\mathbb{R}^N(3\leq N\leq 6)$, and \cite{dpv11,g21,iv08,iv11,wxzc13} for  different potentials.
    When $s\in (0,1)$ and $\Omega=\mathbb{R}^3$, (\ref{P}) has received much attentions,  see \cite{hp2020,l161,l2016,l17,ln17,ln2018,lnp2017,lyy20} and references therein. The previous results are constructing single or multiply peak solutions whose peaks concentrate near the critical points of potentials.

    When $s=1$, $\Omega\subset\mathbb{R}^3$ and $V\equiv1$, the system (\ref{P}) is related to the  follwoing Maxwell-Schr\"{o}dinger system
    \begin{eqnarray}\label{Pms}
    \left\{ \arraycolsep=1.5pt
       \begin{array}{ll}
        -\varepsilon^2 \Delta u+u+\omega\phi u=\gamma u^p,\ \ \ &\ \mbox{in}\ \Omega,\\[2mm]
        -\Delta \phi=4\pi \omega u^2,\ \ \ &\ \mbox{in}\ \Omega,\\[2mm]
        u=\phi=0,\ \ \ \ &\ \mbox{on}\ \partial\Omega,
        \end{array}
    \right.
    \end{eqnarray}
    where $\varepsilon=\frac{\hbar}{\sqrt{2m}}>0$, $m$ is the mass of the particle and $\hbar$ is the Planck's constant, $\omega>0$ denotes the electric charge of the partial and $\gamma>0$, $p>1$. This system arises in Quantum Mechanics: in 1998 Benci and Fortunato \cite{bf98} have proposed it as a model describing the interaction of a changed partial with the electrostatic field.  The unknowns of the system are the wave function $u$ associated to the partial and the electric potential $\phi$. The presence of the nonlinear term in (\ref{Pms}) simulates the interaction effect among many partial. See \cite{m61} for a deeper analysis on the physical motivation of this system.

    The study of the semiclassical limit for Schr\"{o}dinger-Poisson system has been considered, such as \cite{dw05,dw06,dw06cvpde,dw06jde,r05} with $\gamma=\varepsilon^{(p-1)/2}$ and $1<p<\frac{11}{7}$.
    In such papers problem (\ref{Pms}) is studied and it is proved that the radial solutions exhibit some kind of notable concentration behavior: their form consists of very sharp peaks which become highly concentrated when $\varepsilon$ is small. More precisely, when $\Omega=\mathbb{R}^3$, problem (\ref{Pms}) is know to have {\em clustered} solutions, i.e. a combination of several interacting peaks concentrating at the same point as $\varepsilon\to 0^+$ \cite{dw05,dw06cvpde,r05}.
    When $\Omega$ is the unit ball $B_1$ of $\mathbb{R}^3$, D'Aprile and Wei \cite{dw06} constructed a radial solution $(u_\varepsilon,\phi_\varepsilon)$ such that $u_\varepsilon$ concentrates at a distance $(\varepsilon/2)|\log\varepsilon|$ away from the boundary $\partial B_1$ as $\varepsilon\to 0^+$, and with additional conditions about $\omega$, D'Aprile and Wei \cite{dw06jde} also have constructed a radial solution $(u_\varepsilon,\phi_\varepsilon)$ such that $u_\varepsilon$ concentrates on a sphere in the interior of $B_1$ as $\varepsilon\to 0^+$.
    Moreover, in \cite{dw07} {\em clusters} are proved to exist for (\ref{Pms}) with general nonlinearities in a bounded and smooth domain $\Omega\subset\mathbb{R}^3$ and the location of the peaks is identified asymptotically in the terms of the Robin's function of $\Omega$. What's more, when $\Omega\subset\mathbb{R}^3$ is a smooth and bounded domain and $\gamma=\omega=1$ with general nonlinearities, in \cite{dw06pisa} D'Aprile and Wei have proved that (\ref{Pms}) has a least-energy solution $u_\varepsilon$ which develops, as $\varepsilon\to 0^+$, a single spike layer located near the boundary. And when $\Omega=\mathbb{R}^N(3\leq N\leq 6)$ and $\gamma=\omega=1$ with general nonlinearities, in \cite{iv15} Ianni and Vaira have constructed infinite non-radial and sign-changing peak solutions for system (\ref{Pms}) and the peaks are displaced in suitable symmetric configurations and collapse to the same point as $\varepsilon\to 0^+$. For more results about systems, such as FitzHugh-Nagumo type, see \cite{dy1999,dy1,dy2005,dy2006,dy20062,dy20063,ww}.

    On the other hand, consider the following Schr\"{o}dinger equation
    \begin{equation}\label{Psp}
        \varepsilon^{2s}(-\Delta)^{s}u+V(x)u=u^p\quad \mbox{in}\quad \Omega\subset\mathbb{R}^N
    \end{equation}
    with Dirichlet or Neumann boundary conditions, where $s\in (0,1]$, $1<p<\frac{N+2s}{N-2s}$ if $N>2s$ and $1<p<\infty$ if $N\leq2s$, and $V:\Omega\to\mathbb{R}$ is an assigned potential, in order to understand where concentration occurs and what the profile of the solutions should be.
    When $s=1$, there exist solutions with an arbitrary large number of interior peaks near the critical points of $V$ (see \cite{abc97,ams01,df96,df97,df98,jt04,kw00}) or, in the case $V$$\equiv$constant, near the critical points of the distance from the boundary (see \cite{dy99,gpw00,nw,wei1996}), while spikes at the boundary are founded at the critical points of the mean curvature (see e.g. \cite{dfw99,li98,wei97}), and there is an interesting result that if $\mathbb{R}^N\backslash\Omega$ is a bounded open set, i.e. $\Omega$ is an unbounded closed domain, Dancer and Yan \cite{dy2} proved that (\ref{Psp}) has at least one two-peak solution. Moreover, they also showed if $\mathbb{R}^N\backslash\Omega$ is convex, then both of the peaks of any two-peak solution for (\ref{Psp}) must tend to infinity as $\varepsilon\to 0$, and (\ref{Psp}) has no single peak solution if $\mathbb{R}^N\backslash\Omega$ is convex.
    When $s\in (0,1)$ with zero Dirichlet datum, there exist single or multiply peak solutions or cluster solutions which concentrate at the non-degenerate critical points or local extreme points of $V$ (see \cite{ckl2014,cz,ddw2014,ll19,sz2015}) or, in the case $V$$\equiv$constant. However, there are few literatures about peak solutions for standard equation (\ref{Psp}), D\'{a}vila et al. \cite{dddv}  constructed a single peak solution whose peak is strictly in the domain, and see \cite{c18,nxz18} with Neumann type boundary condition.

    Now, let us go on stating our problem. It's well known that when $N\leq 6s$ (see Section \ref{sectrt}),  for each $\varphi, \psi$, let $\Phi[\varphi\psi]$ be the unique solution of the the following problem
    \begin{eqnarray}\label{Pgb}
    \left\{
       \begin{array}{ll}
        (-\Delta)^{s}\phi= \varphi\psi\ \ \ &\ \mbox{in}\ \Omega,\\[2mm]
        \phi=0\ \ \ \ &\ \mbox{in}\ \mathbb{R}^N\setminus \Omega.
        \end{array}
    \right.
    \end{eqnarray}
    Then we can see that (\ref{P}) is equivalent to the following fractional Schr\"{o}dinger problem
    \begin{eqnarray}\label{Pb}
    \left\{
       \begin{array}{ll}
        \varepsilon^{2s}(-\Delta)^{s}u + u+ \Phi[u^2] u=u^p\ \ \ &\ \mbox{in}\ \Omega,\\[2mm]
        u=0\ \ \ \ &\ \mbox{in}\ \mathbb{R}^N\setminus \Omega.
        \end{array}
    \right.
    \end{eqnarray}

    Let us consider the following limiting problem:
    \begin{eqnarray}\label{Pl}
        (-\Delta)^{s}U+U=U^p\   \ \mbox{in}\ \mathbb{R}^N,\ \
        U\in H^s(\mathbb{R}^N),\ \ U(0)=\max_{y\in\mathbb{R}^N}U(y).
    \end{eqnarray}
    The uniqueness, decaying behavior and non-degeneration of the positive solutions to problem (\ref{Pl}) was first obtained by Frank, Lenzmann \cite{fl2013} in $\mathbb{R}$. Later, Frank, Lenzmann and Silvestre \cite{fls2016} obtained the following important result.

    {\bf Theorem~A.}   {\it When $p\in(1,\frac{N+2s}{N-2s})$ for $N>2s$ or $p\in(1,\infty)$ for $N\leq2s$, problem {\rm (\ref{Pl})} has a unique (up to translation) positive solution $U$ which is radial and strictly decreasing in $|x|$. Moreover, there exist $0<\alpha\leq \beta$ such that
        \begin{equation}\label{bu}
        \frac{\alpha}{1+|x|^{N+2s}}\leq U(x) \leq \frac{\beta}{1+|x|^{N+2s}},\ \ \forall x\in \mathbb{R}^N.
        \end{equation}
    Furthermore, $U$ is non-degenerate in the sense that the linear operator defined by
        \begin{equation*}
        L:=(-\Delta)^{s}+U-pU^{p-1}
        \end{equation*}
    satisfies
        \begin{equation}\label{und}
        \ker L\cap L^2(\mathbb{R}^N)={\rm Span}\left\{\frac{\partial U}{\partial x_1},\frac{\partial U}{\partial x_2},\ldots,\frac{\partial U}{\partial x_N}\right\}.
        \end{equation}}
    Then, we begin to construct the approximate solution as follows: denote $U_{\varepsilon,\xi}(x)=U(\frac{x-\xi}{\varepsilon})$ and let $P_{\varepsilon,\Omega}U_{\varepsilon,\xi}$ be the solution of
    \begin{eqnarray}\label{Pp}
    \left\{
       \begin{array}{ll}
        \varepsilon^{2s}(-\Delta)^{s}u+  u=U^p_{\varepsilon,\xi}\ \ \ &\ \mbox{in}\ \Omega,\\[2mm]
        u=0\ \ \ \ &\ \mbox{in}\ \mathbb{R}^N\setminus \Omega.
        \end{array}
    \right.
    \end{eqnarray}

    Our main result of this paper can be stated as following.
    \begin{theorem}\label{thmcs}
    Suppose that $s\in (0,1)$, $2s<N\leq 6s$ 
    and $p\in (1,\frac{N+2s}{N-2s})$. Then there exists $\varepsilon_0>0$ such that for each $\varepsilon\in (0,\varepsilon_0)$, problem {\rm (\ref{Pb})} has a solution $u_{\varepsilon}$ of the form
        \begin{equation}\label{bsu}
        u_{\varepsilon}=P_{\varepsilon,\Omega}U_{\varepsilon,\xi_{\varepsilon}}+\omega_{\varepsilon,\xi_{\varepsilon}},
        \end{equation}
    where $\xi_{\varepsilon}\in\Omega$ satisfies $d(\xi_{\varepsilon},\partial\Omega)\rightarrow 0$,
    $\frac{d(\xi_{\varepsilon},\partial\Omega)}{\varepsilon}\rightarrow +\infty$,
    and $\omega_{\varepsilon,\xi_{\varepsilon}}$ satisfies

    \begin{eqnarray*}
    \int_{\Omega}(\varepsilon^{2s}|(-\Delta)^{\frac{s}{2}}\omega_{\varepsilon,\xi_{\varepsilon}}|^2+ \omega^2_{\varepsilon,\xi_{\varepsilon}})
    =\varepsilon^{N}o(\varepsilon^{\frac{N+4s}{3}}).
    \end{eqnarray*}
    More precisely, $d(\xi_{\varepsilon},\partial\Omega)\in (\varepsilon^{\frac{2}{3}+\theta}, \varepsilon^{\frac{2}{3}-\theta})$ for any small $\theta>0.$
    \end{theorem}

    \begin{remark}\label{remthmfz}\rm
    We remark that the existence of peak solutions for (\ref{Pb}) with $s=1$ and more general nonlinearities, has been obtained by D'Aprile and Wei \cite{dw06pisa}, and our proof is mainly inspired by their work, and also \cite{dy2,dy1}.
    In contrast to previous literatures for the Laplace case, the fractional Laplacian $(-\Delta)^s$ with $0<s<1$ on $\mathbb{R}^N$ is nonlocal, standard techniques for Laplacian may not be carried out directly. Actually, the fractional Laplacian has some features which are essentially different from the Laplacian. For example, the ground state for (\ref{Pl}) with $0<s<1$ decays algebraically at infinity, while that of Laplace problems  decays exponentially at infinity. Hence, in deal with the fractional Laplacian, new techniques are needed to be developed. In particular, thanks to the work of the authors of \cite{drv} who have established the Green's formula corresponding to fractional Laplacian $(-\Delta)^s$,  we will establish new basic estimates and give a precise estimation of the energy functional at the approximate solutions.
    \end{remark}

    This paper is arranged as follows. In Section \ref{sectrt}, we give the necessary functional settings in order to treat our problem and introduce the Lyapunov-Schmidt reduction method as clearly as possible. In Section \ref{sectpe}, we give some preliminary estimates particularly for $\Phi[U^2_{\varepsilon,\xi}]$. As we will see that the main contributions to the energy functional of (\ref{Pb}) are from the term $\int_{\Omega}U^2_{\varepsilon,\xi}\Phi[U_{\varepsilon,\xi}^2]$ as well as the geometry of the domain. In Section \ref{sectcet}, we will calculate  the energy functional $J_\varepsilon(P_{\varepsilon,\Omega}U_{\varepsilon,\xi})$ associated to (\ref{Pb}), particularly for the estimate of $\int_{\Omega}(U_{\varepsilon,\xi})^p\left(U_{\varepsilon,\xi}- P_{\varepsilon,\Omega}U_{\varepsilon,\xi}\right)$.
    Then, in Section \ref{secttl}, we prove some technical lemmas and develop the reduction. Finally, in Section \ref{secteps}, we will prove Theorem \ref{thmcs} and give some remarks. 

    \quad\noindent{\bfseries Notation.} The symbols $c,C$ will denote positive constants whose exact value may change from line to line still being independent of $\varepsilon$.

\section{{\bfseries Functional setting and the reduction method}}\label{sectrt}

    In this section, we first recall some useful information of the fractional order Sobolev spaces.  We refer to \cite{ada,dpv,sh,cz} for more details.
    Consider the Schwartz space $\mathcal{W}$ of rapidly decaying $C^\infty$ functions on $\mathbb{R}^N$. The topology of this space is generated by the seminorms
    \begin{equation*}
    P_k(\varphi)=\sup_{x\in\mathbb{R}^N}(1+|x|^k)\sum_{|\nu|\leq k}\left|D^\nu \varphi(x)\right|,\ \ k=0,1,2\ldots,
    \end{equation*}
    where $\varphi\in \mathcal{W}$. Let $\mathcal{W}'$ be the set of all tempered distributions, which is the topological dual of $\mathcal{W}$.  As usual, for any $\varphi\in \mathcal{W}$, we denote by
    \begin{equation*}
    \mathcal{F}\varphi(\rho)=(2\pi)^{-\frac{N}{2}}\int_{\mathbb{R}^N}e^{-i\rho \cdot x}\varphi(x)dx,
    \end{equation*}
    the Fourier transformation of $\varphi$ and we recall that one can extend $\mathcal{F}$ from $\mathcal{W}$ to $\mathcal{W}'$.

    We define the classical fractional Sobolev space $H^s(\mathbb{R}^N)$,
    \begin{equation*}
    \begin{split}
    H^s(\mathbb{R}^N):=& \left\{u\in L^2(\mathbb{R}^N): \int_{\mathbb{R}^N}(1+|\rho|^{2s})|\mathcal{F}u(\rho)|^2d\rho<\infty \right\}  \\
    =& \left\{u\in L^2(\mathbb{R}^N): \frac{|u(x)-u(y)|}{|x-y|^{\frac{N}{2}+s}}\in L^2(\mathbb{R}^N\times \mathbb{R}^N)\right\},
    \end{split}
    \end{equation*}
    endowed with the norm
    \begin{equation*}
    \|u\|^2_{H^s(\mathbb{R}^N)}=\|u\|^2_{L^2(\mathbb{R}^N)}
    +\frac{C_{N,s}}{2}\int\int_{\mathbb{R}^N\times \mathbb{R}^N}\frac{|u(x)-u(y)|^2}{|x-y|^{N+2s}} dxdy.
    \end{equation*}
    It is clear that $H^s(\mathbb{R}^N)$ is a Hilbert space. We define also the $H^s_0(\Omega)$ space,
    \begin{equation}\label{defh0s}
    H^s_0(\Omega):=\left\{u\in H^s(\mathbb{R}^N): u=0\ \ \mbox{in} \ \ \mathbb{R}^N\backslash\Omega\right\},
    \end{equation}
    endowed with the norm
    \begin{equation*}
    \|u\|^2_{H^s_0(\Omega)}=\|u\|^2_{L^2(\Omega)}
    +\int\int_{\mathcal{D}_{\Omega}}\frac{|u(x)-u(y)|^2}{|x-y|^{N+2s}} dxdy,
    \end{equation*}
    where the set $\mathcal{D}_{\Omega}$ is given by $\mathcal{D}_{\Omega}:=(\mathbb{R}^N\times \mathbb{R}^N)\backslash (\mathcal{C}\Omega\times\mathcal{C}\Omega),\ \ \mathcal{C}\Omega=\mathbb{R}^N\backslash \Omega.$ Denote
    \begin{equation}\label{defuv}
    \langle u,v \rangle_{\varepsilon}=\int_{\mathbb{R}^N}\left(\varepsilon^{2s}(-\Delta)^{\frac{s}{2}}u(-\Delta)^{\frac{s}{2}}v+  uv\right)dx,
    \end{equation}
    and $\|u\|^2_{\varepsilon}=\langle u,u \rangle_{\varepsilon}$. Let $\mathbb{H}$ be the completion of the space $C^{\infty}_0(\Omega)$ with respect to the norm $\|\cdot\|_\varepsilon$ defined as previous. Noticing that for all $u, v\in \mathbb{H}$, we have
    \begin{equation*}
    \int_{\Omega}\left(\varepsilon^{2s}v(-\Delta)^{s}u +  uv\right)dx
    =\int_{\mathbb{R}^N}\left(\varepsilon^{2s}(-\Delta)^{\frac{s}{2}}u(-\Delta)^{\frac{s}{2}}v+  uv\right)dx.
    \end{equation*}
    Thus
    \begin{equation}\label{defnuv}
    \|u\|^2_{\varepsilon}=\int_{\Omega}\left(\varepsilon^{2s}|(-\Delta)^{\frac{s}{2}}u|^2 +  u^2\right)dx,
    \end{equation}
    that justify in some sense our choice of the norm $\|\cdot\|_\varepsilon$. Since the embedding $H^s_0(\Omega_{\varepsilon,\xi})\hookrightarrow L^q(\Omega_{\varepsilon,\xi})$ is continuous for $2\leq q\leq\frac{2N}{N-2s}$, then for each $u\in \mathbb{H}$, and denote $\tilde{u}(x)=u(\varepsilon x+\xi)$, we have
    \begin{equation}\label{deflqfh}
    \begin{split}
    \int_{\Omega}|u|^{q} dx
    = & \varepsilon^N\int_{\Omega_{\varepsilon,\xi}}|\tilde{u}|^{q} dx
    \leq  C \varepsilon^N\left(\int_{\Omega_{\varepsilon,\xi}}\left(|(-\Delta)^{\frac{s}{2}}\tilde{u}|^2 +  |\tilde{u}|^2\right) dx\right)^{\frac{q}{2}}  \\
    = & C \varepsilon^{N}\left(\varepsilon^{-N}\int_{\Omega}\left(\varepsilon^{2s}|(-\Delta)^{\frac{s}{2}}u|^2 +  |u|^2\right) dx\right)^{\frac{q}{2}}
    =  C \varepsilon^{(1-\frac{q}{2})N}\|u\|^q_{\varepsilon},
    \end{split}
    \end{equation}
    that is, $|u|_q\leq C\varepsilon^{(\frac{1}{q}-\frac{1}{2})N}\|u\|_{\varepsilon}$ for all $2\leq q\leq\frac{2N}{N-2s}$. Moreover, $|u|_q \leq C\|u\|_{\varepsilon}$ for all $q\in (1,2]$.

    For $\varphi,\psi\in H^s(\mathbb{R}^N)$, the linear functional $\mathcal{L}_{\varphi,\psi}$ defined in $\mathbb{H}$ by
    \begin{equation*}
    \mathcal{L}_{\varphi,\psi} (v)=\int_{\Omega}\varphi\psi vdx.
    \end{equation*}
    By the H\"{o}lder's inequality, we have
    \begin{equation}\label{deflu}
    \begin{split}
    |\mathcal{L}_{\varphi,\psi} (v)|
    \leq & \left(\int_{\Omega}|\varphi\psi|^{\frac{2N}{N+2s}}dx\right)^{\frac{N+2s}{2N}}
    \left(\int_{\Omega}|v|^{\frac{2N}{N-2s}}dx\right)^{\frac{N-2s}{2N}}  \\
    \leq & \left(\int_{\Omega}|\varphi|^{\frac{4N}{N+2s}}dx\right)^{\frac{N+2s}{4N}}
    \left(\int_{\Omega}|\psi|^{\frac{4N}{N+2s}}dx\right)^{\frac{N+2s}{4N}}
    \left(\int_{\Omega}|v|^{\frac{2N}{N-2s}}dx\right)^{\frac{N-2s}{2N}}  \\
    \leq & C \varepsilon^{-s}\|\varphi\|_{H^s(\mathbb{R}^N)}\|\psi\|_{H^s(\mathbb{R}^N)} \|v\|_\varepsilon,
    \end{split}
    \end{equation}
    where we used (\ref{deflqfh}) for $\frac{4N}{N+2s}\leq \frac{2N}{N-2s}$, i.e. $N\leq 6s$. Hence, by the Lax-Milgram theorem, there exists a unique $\Phi[\varphi\psi]\in \mathbb{H}$ such that
    \begin{equation}\label{deflue}
    \int_{\Omega}(-\Delta)^{\frac{s}{2}}\Phi[\varphi\psi](-\Delta)^{\frac{s}{2}}vdx=\int_{\Omega}\varphi\psi vdx,\quad v\in \mathbb{H},
    \end{equation}
    that is $\Phi[\varphi\psi]$ is a weak solution of
    \begin{equation*}
    (-\Delta)^{s}\phi=\varphi\psi\quad\mbox{in}\quad\Omega,\quad \phi\in \mathbb{H}.
    \end{equation*}
    From (\ref{deflu}) and (\ref{deflue}), we have that for each $u\in \mathbb{H}$,
    \begin{equation*}
    \begin{split}
    \|\Phi[u^2]\|^2_\varepsilon \leq & C \int_{\Omega}|(-\Delta)^{\frac{s}{2}}\Phi[u^2]|^2dx
    =  C \int_{\Omega}\Phi[u^2]u^2dx
    \leq   C \varepsilon^{-\frac{N}{2}}\|u\|^2_\varepsilon \|\Phi[u^2]\|_\varepsilon,  \\
    \end{split}
    \end{equation*}
    that is
    \begin{equation*}
    \begin{split}
    \|\Phi[u^2]\|_\varepsilon \leq & C \varepsilon^{-\frac{N}{2}}\|u\|^2_\varepsilon,\quad \mbox{if}\quad N\leq 6s,\\
    \end{split}
    \end{equation*}
    Thus
    \begin{equation}\label{epuud}
    \begin{split}
    \int_{\Omega}\Phi[u^2]u^2dx\leq & C \varepsilon^{-N}\|u\|^4_\varepsilon,\quad \mbox{if}\quad N\leq 6s.\\
    \end{split}
    \end{equation}
    Moreover, let us consider the problem
    \begin{equation*}
    (-\Delta)^{s}W=\varphi\psi\quad\mbox{in}\quad\mathbb{R}^N,\quad W\in H^s(\mathbb{R}^N),
    \end{equation*}
    then by the representation formula holds
    \begin{equation}\label{defprf2}
    W[\varphi\psi](x)=c_{N,s} \int_{\mathbb{R}^N}\frac{\varphi(y)\psi(y)}{|x-y|^{N-2s}}dy, \ \ x\in\mathbb{R}^N,
    \end{equation}
    which is called $s$-Riesz potential, where $c_{N,s}$ is a constant depending only on $N$ and $s$. Thus, by the comparison theorem, we have
    \begin{equation}\label{ctpw}
    \Phi[\varphi\psi]\leq \Phi[|\varphi\psi|]\leq W[|\varphi\psi|]\quad\mbox{for all}\quad\varphi,\psi\in H^s(\mathbb{R}^N).
    \end{equation}

    \begin{proposition}\label{prohlsi}
    (Hardy-Littlewood-Sobolev inequality \cite{ll01}) Let $r,t>1$ and $0<\mu<N$ with $\frac{1}{r}+\frac{1}{t}+\frac{\mu}{N}=2$, $u\in L^r(\mathbb{R}^N)$ and $v\in L^t(\mathbb{R}^N)$. There exists a sharp constant $C(r,t,\mu,N)$, independent of $u$ and $v$, such that
    \begin{equation*}\label{hlsi}
    \int_{\mathbb{R}^N}\left[\frac{1}{|x|^\mu}\ast u(x)\right]v(x)dx\leq C(r,t,\mu,N)|u|_r|v|_t.
    \end{equation*}
    \end{proposition}
    Thus, by the above inequality and (\ref{deflqfh}), (\ref{ctpw}), we get that when $N\leq 6s$, for all $u\in \mathbb{H}$,
    \begin{equation}\label{epuu}
    \begin{split}
    \int_{\Omega}\Phi[u^2]u^2dx
    \leq & \int_{\Omega}W[u^2]u^2dx
    =c_{N,s} \int_{\Omega}\int_{\Omega}\frac{u^2(y)u^2(x)}{|x-y|^{N-2s}}dydx \\
    \leq & c_{N,s} \left(\int_{\Omega}u^\frac{4N}{N+2s}dx\right)^{\frac{N+2s}{N}}
    \leq C \varepsilon^{2s-N}\|u\|^4_\varepsilon.\\
    \end{split}
    \end{equation}
    Substituting $\Phi[u^2]$ in (\ref{P}), then we can see that (\ref{P}) with $V\equiv1$ is equivalent to the following fractional Schr\"{o}dinger problem:
    \begin{eqnarray*}
    \left\{
       \begin{array}{ll}
        \varepsilon^{2s}(-\Delta)^{s}u + u+ \Phi[u^2] u=u^p\ \ \ &\ \mbox{in}\ \Omega,\\[2mm]
        u=0\ \ \ \ &\ \mbox{in}\ \mathbb{R}^N\setminus \Omega.
        \end{array}
    \right.
    \end{eqnarray*}
    The energy functional associated with (\ref{Pb}) is
    \begin{equation}\label{efPb}
    \begin{split}
    J_{\varepsilon}(u)= & \frac{\varepsilon^{2s}C_{N,s}}{2}\int\int_{\mathcal{D}_{\Omega}}\frac{|u(x)-u(y)|^2}{|x-y|^{N+2s}} dxdy
    +\frac{1}{4}\int_{\Omega}\Phi[u] u^2 dx-\int_{\Omega}\left( \frac{u^{p+1}}{p+1}- \frac{u^{2}}{2}\right)dx  \\
    = & \frac{1}{2}\|u\|^2_{\varepsilon}+\frac{1}{4}\int_{\Omega}\Phi[u^2] u^2 dx-\frac{1}{p+1}\int_{\Omega}u^{p+1}dx.
    \end{split}
    \end{equation}
    Define
    \begin{equation}\label{defe}
    E_{\varepsilon,\xi}:=\left\{\omega\in \mathbb{H}: \left\langle \omega,\frac{\partial P_{\varepsilon,\Omega}U_{\varepsilon,\xi}}{\partial \xi_i}\right\rangle_{\varepsilon}=0,\ i=1,2,\ldots, N\right\},
    \end{equation}
    where $\xi=(\xi_1,\ldots,\xi_N)\in\mathbb{R}^N$.
    Let $D_{\varepsilon,R}:=\{\xi:\xi\in\Omega,\ R\varepsilon \leq d(\xi,\partial\Omega)\leq 1/R\}$, where $R>0$ is a large constant.
    In Section \ref{secteps}, we will choose $\xi\in D_{\varepsilon,R}$ for suitable domain $D_{\varepsilon,R}$ such that $(P_{\varepsilon,\Omega}U_{\varepsilon,\xi}+\omega_{\varepsilon,\xi})$ is a solution of (\ref{Pb}). Let
    \begin{equation*}
    K(\xi,\omega)=J_\varepsilon\left(P_{\varepsilon,\Omega}U_{\varepsilon,\xi}+\omega\right),\ \xi\in D_{\varepsilon,R},\ \omega\in E_{\varepsilon,\xi}.
    \end{equation*}
    We expand $K(\xi,\omega)$ near $\omega=0$ as
    \begin{equation*}
    K(\xi,\omega)=K(\xi,0)+L_{\varepsilon,\xi}(\omega)+\frac{1}{2}Q_{\varepsilon,\xi}(\omega,\omega)+R_\varepsilon(\omega),
    \end{equation*}
    where
    \begin{equation}\label{rl}
    \begin{split}
    L_{\varepsilon,\xi}(\omega)=\langle P_{\varepsilon,\Omega}U_{\varepsilon,\xi}, \omega \rangle_\varepsilon
    + \int_{\Omega}\Phi[P^2_{\varepsilon,\Omega}U_{\varepsilon,\xi}]P_{\varepsilon,\Omega}U_{\varepsilon,\xi} \omega  dx
    - \int_{\Omega}(P_{\varepsilon,\Omega}U_{\varepsilon,\xi})^p\omega dx,
    \end{split}
    \end{equation}

    \begin{equation}\label{rq}
    \begin{split}
    Q_{\varepsilon,\xi}(\omega,\eta)
    = & \langle \omega,\eta \rangle_\varepsilon
    - p\int_{\Omega}(P_{\varepsilon,\Omega}U_{\varepsilon,\xi})^{p-1}\omega\eta dx
    +\int_{\Omega}\Phi[P^2_{\varepsilon,\Omega}U_{\varepsilon,\xi}]\omega\eta  dx  \\
    & +2\int_{\Omega}\Phi[\omega P_{\varepsilon,\Omega}U_{\varepsilon,\xi}]
    P_{\varepsilon,\Omega}U_{\varepsilon,\xi}(x)\eta(x)  dx,\ \ \forall\omega,\ \eta\in\mathbb{H},
    \end{split}
    \end{equation}
    and
    \begin{equation}\label{rr}
    \begin{split}
    R_{\varepsilon}(\omega)
    = & \int_{\Omega}\Phi[\omega^2]P_{\varepsilon,\Omega}U_{\varepsilon,\xi}\omega dx
    +\frac{1}{4}\int_{\Omega}\Phi[\omega^2]\omega^2 dx
    -\frac{1}{p+1}\int_{\Omega}(P_{\varepsilon,\Omega}U_{\varepsilon,\xi}+\omega)^{p+1}dx  \\
    & + \frac{1}{p+1}\int_{\Omega}(P_{\varepsilon,\Omega}U_{\varepsilon,\xi})^{p+1}dx
    +\int_{\Omega}(P_{\varepsilon,\Omega}U_{\varepsilon,\xi})^p\omega dx
    +\frac{p}{2}\int_{\Omega}(P_{\varepsilon,\Omega}U_{\varepsilon,\xi})^{p-1}\omega^2 dx.
    \end{split}
    \end{equation}
    We will prove in Lemma \ref{lemel} that $L_{\varepsilon,\xi}:E_{\varepsilon,\xi}\rightarrow \mathbb{R}$ is a bounded linear operator. In particular, by the Riesz representation theorem there is an $l_{\varepsilon,\xi}\in E_{\varepsilon,\xi}$ such that $\langle l_{\varepsilon,\xi}, \omega\rangle_\varepsilon     =L_{\varepsilon,\xi}(\omega)$. In Lemma \ref{lemqb}, we will prove
    \begin{equation*}
    |Q_{\varepsilon,\xi}(\omega,\omega)|\leq C\|\omega\|_\varepsilon^2,
    \end{equation*}
    where $C$ is a constant independent of $\varepsilon$. In particular, there is a bounded linear operator $\mathcal{A}_{\varepsilon,\xi}: E_{\varepsilon,\xi}\rightarrow E_{\varepsilon,\xi}$ such that $\langle \mathcal{A}_{\varepsilon,\xi}\omega, \eta\rangle_\varepsilon     =Q_{\varepsilon,\xi}(\omega,\eta)$. Thus, finding a critical point for $K(\xi,\omega)$ in $E_{\varepsilon,\xi}$ is equivalent to solving the following problem in $E_{\varepsilon,\xi}$:
    \begin{equation}\label{re}
    l_{\varepsilon,\xi}+\mathcal{A}_{\varepsilon,\xi}\omega+R'_\varepsilon(\omega)=0.
    \end{equation}
    We will prove in Lemma \ref{lemeqw} that the operator $\mathcal{A}_{\varepsilon,\xi}$ is invertible in $E_{\varepsilon,\xi}$. So we see that $l_{\varepsilon,x}+\mathcal{A}_{\varepsilon,x}\omega=0$ has a unique solution in $E_{\varepsilon,\xi}$ for each fixed $\xi\in D_{\varepsilon,R}$. In Lemma \ref{profu}, we will prove that if $\omega$ belongs to a suitable set, the term $R'_\varepsilon(\omega)$ is a small perturbation term in (\ref{re}). Thus, we can use the contraction mapping theorem to prove that (\ref{re}) has a unique solution for each fixed $\xi\in D_{\varepsilon,R}$.

\section{{\bfseries Some preliminary estimates}}\label{sectpe}

    In this section, we will give some preliminary estimates which are crucial for our result.

    Firstly, let us analysis the properties of $U_{\varepsilon,\xi}$ and $P_{\varepsilon,\Omega}U_{\varepsilon,\xi}$. Fix $\xi\in D_{\varepsilon,R}$, $\xi=(\xi_1,\xi_2,\ldots,\xi_N)$. From (\ref{bu}), we have
    \begin{equation}\label{bub}
        U_{\varepsilon,\xi}(x) \leq C\left(\frac{\varepsilon}{d(\xi,\Omega)}\right)^{N+2s},\ \ \forall x\in \mathbb{R}^N\backslash\Omega,
    \end{equation}
    for some $C>0$. Let $\eta_{\varepsilon,\xi}=U_{\varepsilon,\xi}-P_{\varepsilon,\Omega}U_{\varepsilon,\xi},$
    then from (\ref{Pp}) we have
    \begin{eqnarray}\label{defetae}
    \left\{ \arraycolsep=1.5pt
       \begin{array}{ll}
        \varepsilon^{2s}(-\Delta)^{s}\eta_{\varepsilon,\xi}+  \eta_{\varepsilon,\xi}=0\ \ \ &\ \mbox{in}\ \Omega,\\[2mm]
        \eta_{\varepsilon,\xi}=U_{\varepsilon,\xi}\ \ \ \ &\ \mbox{in}\ \mathbb{R}^N\backslash \Omega.
        \end{array}
    \right.
    \end{eqnarray}

    \begin{lemma}\label{lemupu}
    There exists a positive constant $C$ such that for any $x\in\mathbb{R}^N$, it holds that
    \begin{equation}\label{upu}
    0\leq \eta_{\varepsilon,\xi}\leq C \left(\frac{\varepsilon}{d(\xi,\Omega)}\right)^{N+2s}.
    \end{equation}
    \end{lemma}

    \begin{proof}
    By using the estimation (\ref{bub}), we have that $\eta_{\varepsilon,\xi}\leq U_{\varepsilon,\xi}\leq C\left(\frac{\varepsilon}{d(\xi,\Omega)}\right)^{N+2s}$ in $\mathbb{R}^N\backslash\Omega$. Hence, by the maximum principle, we can directly get (\ref{upu}).
    \end{proof}

    \begin{lemma}\label{lemeguc}
    There is $C>0$ such that for every $x\in\mathbb{R}^N$,
    \begin{equation}\label{estgru1}
    \Phi[U^2_{\varepsilon,\xi}](x)\leq \frac{C \varepsilon^{2s}}{1+\left|\frac{x-\xi}{\varepsilon}\right|^{N-2s}},
    \end{equation}
    and
    \begin{equation}\label{estgru2}
    \Phi[U_{\varepsilon,\xi}](x)\leq \frac{C \varepsilon^{2s}}{1+\left|\frac{x-\xi}{\varepsilon}\right|^{N-2s}}.
    \end{equation}
    \end{lemma}

    \begin{proof}
    We only prove (\ref{estgru1}), and the proof of (\ref{estgru2}) is similar.
    Noticing that $\Phi[U^2_{\varepsilon,\xi}](\varepsilon x+\xi)=\varepsilon^{2s}v_\varepsilon(x)$ and $v_\varepsilon$ satisfies 
    \begin{eqnarray*}
    \left\{
       \begin{array}{ll}
        (-\Delta)^{s}v_\varepsilon=U^2\ \ \ &\ x\in\Omega_{\varepsilon,\xi},\\[2mm]
        v_\varepsilon=0\ \ \ \ &\ x\in\mathbb{R}^N\backslash \Omega_{\varepsilon,\xi}.
        \end{array}
    \right.
    \end{eqnarray*}
    Let $W$ be the solution of the following equation
    \begin{equation}\label{gps}
    (-\Delta)^{s}W=U^2,\ \ W\in H^s(\mathbb{R}^N).
    \end{equation}
    It is well known that
    \begin{equation}\label{gpsj}
    W(x)=c_{N,s}\int_{\mathbb{R}^N}\frac{U^2(y)}{|y-x|^{N-2s}}dy, \quad x\in \mathbb{R}^N,
    \end{equation}
    where $c_{N,s}$ is a constant depending only on $N$ and $s$. From \cite[Lemma 5.1]{dddv},  we have
    \begin{equation}\label{gpsdx}
    \frac{C_1}{1+|x|^{N-2s}}\leq W(x)\leq \frac{C_2}{1+|x|^{N-2s}},
    \end{equation}
    for some $C_2\geq C_1>0.$ Thus
    \begin{equation*}
    \begin{split}
    W(0)
    = & c_{N,s}\int_{\mathbb{R}^N}\frac{U^2(y)}{|y|^{N-2s}}dy \\
    \leq & \int_{\mathbb{R}^N}\frac{C}{|y|^{N-2s}(1+|y|^{N+2s})^2}dy \\
    = & \int_{B_{1}(0)}\frac{C}{|y|^{N-2s}(1+|y|^{N+2s})^2}dy+ \int_{\mathbb{R}^N\backslash B_{1}(0)}\frac{C}{|y|^{N-2s}(1+|y|^{N+2s})^2}dy \\
    \leq & \int_{B_{1}(0)}\frac{C}{|y|^{N-2s}}dy+ \int_{\mathbb{R}^N\backslash B_{1}(0)}\frac{C}{(1+|y|^{N+2s})^2}dy
    \leq  C',
    \end{split}
    \end{equation*}
    and also we can get $W(0)\geq C$. Then fix $x\in\mathbb{R}^N$ for $|x|\geq 2$ large and $t:=|x|/2$, and observe that if $y\in B_{t}(x)$ then $|y|\geq |x|-|x-y|\geq |x|-t=\frac{|x|}{2}.$ As a consequence,
    \begin{equation*}
    \begin{split}
    \int_{B_{t}(x)}\frac{U^2(y)}{|y-x|^{N-2s}}dy
    \leq & \int_{B_{t}(x)}\frac{C}{|y-x|^{N-2s}(1+|y|^{N+2s})^2}dy \\
    \leq & \int_{B_{t}(x)}\frac{C}{|y-x|^{N-2s}(1+(|x|/2)^{N+2s})^2}dy \\
    \leq & \frac{C}{1+|x|^{2(N+2s)}}\int_{B_{t}(0)}\frac{1}{|z|^{N-2s}}dz \\
    \leq & \frac{C'|t|^{2s}}{1+|x|^{2(N+2s)}}
    \leq  \frac{C''}{1+|x|^{2N+2s}}.
    \end{split}
    \end{equation*}
    Moreover, if $y\in \mathbb{R}^N\backslash B_{t}(x)$ then $|y-x|\geq t=\frac{|x|}{2},$ thus
    \begin{equation*}
    \begin{split}
    \int_{\mathbb{R}^N\backslash B_{t}(x)}\frac{U^2(y)}{|y-x|^{N-2s}}dy
    \leq & \int_{\mathbb{R}^N\backslash B_{t}(x)}\frac{C}{(|x|/2)^{N-2s}(1+|y|^{N+2s})^2}dy \\
    \leq & \frac{C'}{|x|^{N-2s}}\int_{\mathbb{R}^N}\frac{1}{(1+|y|^{N+2s})^2}dy \\
    \leq & \frac{C''}{1+|x|^{N-2s}}.
    \end{split}
    \end{equation*}
    Furthermore, if $y\in B_{t}(0)$ then $\frac{3|x|}{2}=3t\geq |y-x|\geq t=\frac{|x|}{2},$ thus
    \begin{equation*}
    \begin{split}
    \int_{B_{t}(0)}\frac{U^2(y)}{|y-x|^{N-2s}}dy
    \geq & \int_{B_{t}(0)}\frac{C}{(3|x|/2)^{N-2s}(1+|y|^{N+2s})^2}dy \\
    \geq & \frac{C'}{|x|^{N-2s}}\int_{B_{1}(0)}\frac{1}{(1+|y|^{N+2s})^2}dy \\
    \geq & \frac{C''}{1+|x|^{N-2s}}.
    \end{split}
    \end{equation*}
    Thus, (\ref{gpsdx}) holds.
    By the comparison theorem, we have $0<v_\varepsilon\leq W.$ Thus we have the following estimation of $v_\varepsilon$:
    \begin{equation*}
    v_\varepsilon(x)\leq \frac{C}{1+|x|^{N-2s}},\ \ \forall\ x\in \mathbb{R}^N,
    \end{equation*}
    i.e., there exists $C>0$ such that (\ref{estgru1}) holds.
    \end{proof}

    Finally, at the end of this section, let us make a more accurate estimation of $\Phi[U^2_{\varepsilon,\xi}]$.

    \begin{lemma}\label{lemegu}
    We have
    \begin{equation*}
    \begin{split}
    \Phi[U^2_{\varepsilon,\xi}](x)= & \varepsilon^{2s}W\left(\frac{x-\xi}{\varepsilon}\right)
    -\varepsilon^{N}B H(\xi,x)
    +O\left(\mu_{\varepsilon,\xi}
    \right),
    \end{split}
    \end{equation*}
    where $W$ is given in (\ref{gpsj}), $B=\int_{\mathbb{R}^N} U^2$, $H(\xi,x)$ is the regular part of the Green function of the operator $(-\Delta)^s$ with the Dirichlet boundary condition, and $\mu_{\varepsilon,\xi}$ will be defined in (\ref{defmu}).
    \end{lemma}

    \begin{proof}
    Let $S(x,\xi)$ be the solution of
    \begin{eqnarray}\label{lemeguPs}
    \left\{
       \begin{array}{ll}
        (-\Delta)^{s}S=\delta_\xi\ \ \ &\ \mbox{in}\ \mathbb{R}^N,\\[2mm]
        S(x,\xi)\rightarrow 0\ \ \ \ &\ \mbox{as}\ |x|\rightarrow \infty,
        \end{array}
    \right.
    \end{eqnarray}
    where $\delta_\xi$ is the Dirac measure. By a direct calculation, we know that $S(x,\xi)$ behaves like $\frac{1}{|x-\xi|^{N-2s}}$, i.e.
    \begin{equation*}
    S(x,\xi)\simeq \frac{C}{|x-\xi|^{N-2s}}.
    \end{equation*}
    Then
    \begin{equation*}
    \Phi[U^2_{\varepsilon,\xi}](x)=\int_{\Omega}\left(S(y,x)-H(y,x)\right)U^2_{\varepsilon,\xi}(y) dy.
    \end{equation*}
    Since $H(x,\xi)$ is the regular part of the Green function of the operator $(-\Delta)^s$ with the Dirichlet boundary condition, i.e.
    \begin{eqnarray}\label{lemeguPh}
    \left\{
       \begin{array}{ll}
        (-\Delta)^{s}_x H(x,\xi)=0\ \ \ &\ \mbox{in}\ \Omega,\\[2mm]
        H(x,\xi)=S(x,\xi)\ \ \ \ &\ \mbox{in}\ \mathbb{R}^N\backslash \Omega,
        \end{array}
    \right.
    \end{eqnarray}
    we have
    \begin{equation}\label{lemeguh}
    |D^l_x H(x,\xi)|\leq \frac{C}{d(\xi,\partial \Omega)^{N-2s+l}},\ \ l=0,1,2.
    \end{equation}
    By using the mean value theorem, we have
    \begin{align}\label{lemeguhu}
    & \int_{\Omega}H(y,x)U^2_{\varepsilon,\xi}(y) dy \nonumber\\
    = & \varepsilon^N \int_{\Omega_{\varepsilon,\xi}}H(\varepsilon z+\xi,x)U^2(z) dz  \nonumber\\
    = & \varepsilon^N \int_{\Omega_{\varepsilon,\xi}}\left[H(\xi,x)+\sum^{N}_{k=1}\varepsilon D_{k}H(\xi,x)\cdot z_k+O\left(\sum^N_{i,j>0}\varepsilon^2 D_{ij}H(\xi,x)\cdot z_i z_j\right)\right]U^2(z) dz  \nonumber\\
    = & \varepsilon^N H(\xi,x)\left(B-\int_{\mathbb{R}^N\backslash\Omega_{\varepsilon,\xi}}U^2(z) dz\right) +\varepsilon^{N+1}\sum^{N}_{k=1}D_{k}H(\xi,x)\int_{\Omega_{\varepsilon,\xi}} z_kU^2(z) dz\nonumber\\
      &
        +O\left(\varepsilon^{N+2}\sum^N_{i,j>0} D_{ij}H(\xi,x)\int_{\Omega_{\varepsilon,\xi}} z_i z_jU^2(z) dz\right)
    \end{align}
    where $\Omega_{\varepsilon,\xi}=\{x: \varepsilon x+\xi\in \Omega\}$.
    Since $B_{\frac{d(\xi,\partial\Omega)}{\varepsilon}}(0)\subset \Omega_{\varepsilon,\xi}$ and by the behavior of $U$, we obtain
    \begin{equation*}
    \begin{split}
    \int_{\mathbb{R}^N\backslash\Omega_{\varepsilon,\xi}}U^2(z) dz\leq & \int_{\mathbb{R}^N\backslash B_{\frac{d(\xi,\partial\Omega)}{\varepsilon}}(0)}U^2(z) dz  \\
    \leq & C\int_{\mathbb{R}^N\backslash B_{\frac{d(\xi,\partial\Omega)}{\varepsilon}}(0)}\frac{1}{(1+|z|^{N+2s})^2} dz \\
    = & O\left(\frac{\varepsilon}{d(\xi,\partial\Omega)}\right)^{N+4s}.
    \end{split}
    \end{equation*}
    Since $U(z)=U(|z|)$ and
    \begin{equation*}
    |\sum^{N}_{k=1}z_k|\leq N|\sum^{N}_{k=1}z^2_k|^{1/2}=N|z|,\quad \forall z=(z_1,\ldots,z_N)\in \mathbb{R}^N,
    \end{equation*}
    we obtain
    \begin{equation*}
    \begin{split}
    \left|\varepsilon^{N+1}\sum^{N}_{k=1}D_{k}H(\xi,x)\int_{\Omega_{\varepsilon,\xi}} z_kU^2(z) dz\right|
    = & \left|\varepsilon^{N+1}\sum^{N}_{k=1}D_{k}H(\xi,x)\int_{\Omega_{\varepsilon,\xi}\backslash B_{\frac{d(\xi,\partial\Omega)}{\varepsilon}}(0)} z_kU^2(z) dz\right|  \\
    \leq & \frac{C\varepsilon^{N+1}}{d(\xi,\partial\Omega)^{N-2s+1}}\int_{\Omega_{\varepsilon,\xi}\backslash B_{\frac{d(\xi,\partial\Omega)}{\varepsilon}}(0)} \frac{|z|}{(1+|z|^{N+2s})^2} dz \\
    = & O\left(\frac{\varepsilon^{2N+4s}}{d(\xi,\partial\Omega)^{2N+2s}}\right).
    \end{split}
    \end{equation*}
    Moreover, since $\Omega\subset \mathbb{R}^N$ is bounded, we can choose $r>0$ which is independent of $\varepsilon$ such that $\Omega_{\varepsilon,\xi}\subset B_{\frac{r}{\varepsilon}}(0)$, then we have
    \begin{equation*}
    \begin{split}
    \varepsilon^{N+2}\sum^N_{i,j>0}D_{ij}H(\xi,x)\int_{\Omega_{\varepsilon,\xi}} z_i z_jU^2(z) dz
    \leq & \frac{C\varepsilon^{N+2}}{d(\xi,\partial\Omega)^{N-2s+2}}\int_{\Omega_{\varepsilon,\xi}} \frac{|z|^2}{(1+|z|^{N+2s})^2} dz \\
    \leq & \frac{C\varepsilon^{N+2}}{d(\xi,\partial\Omega)^{N-2s+2}}\int_{B_{\frac{r}{\varepsilon}}(0)} \frac{|z|^2}{1+|z|^{2N+4s}} dz \\
    = & O\left(\frac{\varepsilon^{N+2}|\log\varepsilon|}{d(\xi,\partial\Omega)^{N-2s+2}}\right),
    \end{split}
    \end{equation*}
    if $N+4s-2\geq0$ i.e. $N\geq2, s\in(0,1)$ or $N=1, s\in[\frac{1}{4},\frac{1}{2})$. Otherwise, if $N+4s-2<0$ i.e. $N=1, s\in(0,\frac{1}{4})$, we have
    \begin{equation*}\label{lemeguij0x}
    \begin{split}
    \varepsilon^{N+2}\sum^N_{i,j>0}D_{ij}H(\xi,x)\int_{\Omega_{\varepsilon,\xi}} z_i z_jU^2(z) dz
    = & O\left(\frac{\varepsilon^{2N+4s}}{d(\xi,\partial\Omega)^{N-2s+2}}\right).
    \end{split}
    \end{equation*}
    Define
    \begin{eqnarray}\label{defmu}
    \mu_{\varepsilon,\xi}=
    \left\{ \arraycolsep=1.5pt
       \begin{array}{ll}
        \frac{\varepsilon^{N+2}|\log\varepsilon|}{d(\xi,\partial\Omega)^{N-2s+2}},\ \ &{\rm if}\ \ N+4s-2\geq0;\\[3mm]
        \frac{\varepsilon^{2N+4s}}{d(\xi,\partial\Omega)^{N-2s+2}},\ \ &{\rm if}\ \ N+4s-2<0.
        \end{array}
    \right.
    \end{eqnarray}
    Therefore,
    \begin{equation}\label{muxx}
    \begin{split}
    \varepsilon^N H(\xi,x)\int_{\mathbb{R}^N\backslash\Omega_{\varepsilon,\xi}}U^2(z) dz
    =O\left(\frac{\varepsilon^{2N+4s}}{d(\xi,\partial\Omega)^{2N+2s}}\right)
    = & O\left(\mu_{\varepsilon,\xi}\right),  \\
    \varepsilon^{N+1}\sum^{N}_{k=1}D_{k}H(\xi,x)\int_{\Omega_{\varepsilon,\xi}} z_kU^2(z) dz= & O\left(\mu_{\varepsilon,\xi}\right),  \\
    \varepsilon^{N+2}\sum^N_{i,j}D_{ij}H(\xi,x)\int_{\Omega_{\varepsilon,\xi}} z_i z_jU^2(z) dz
    = & O\left(\mu_{\varepsilon,\xi}\right).
    \end{split}
    \end{equation}
    On the other hand,
    \begin{equation}\label{lemegusu}
    \begin{split}
    \int_{\Omega}S(y,x)U^2_{\varepsilon,\xi}(y) dy
    = & \varepsilon^N \int_{\Omega_{\varepsilon,\xi}}S(\varepsilon z+\xi,x)U^2(z) dz  \\
    = & \varepsilon^N \int_{\mathbb{R}^N}S(\varepsilon z+\xi,x)U^2(z) dz+ O\left(\frac{\varepsilon^{2N+4s}}{d(\xi,\partial\Omega)^{2N+2s}}\right) \\
    = & \varepsilon^N \int_{\mathbb{R}^N}S(\varepsilon z+\xi,x)U^2(z) dz+ O\left(\mu_{\varepsilon,\xi}\right).
    \end{split}
    \end{equation}
    But $\varepsilon^{N-2s}S(\varepsilon z+\xi,x)$ is a solution of
    \begin{eqnarray*}
    \left\{
       \begin{array}{ll}
        (-\Delta)^{s}w=\delta_{(x-\xi)/\varepsilon}\ \ \ &\ \mbox{in}\ \mathbb{R}^N,\\[2mm]
        w(x)\rightarrow 0\ \ \ \ &\ \mbox{as}\ |x|\rightarrow \infty.
        \end{array}
    \right.
    \end{eqnarray*}
    As a result,
    \begin{equation}\label{lemegusu2}
    \begin{split}
    \varepsilon^{N-2s} \int_{\mathbb{R}^N}S(\varepsilon z+\xi,x)U^2(z) dz=W\left(\frac{x-\xi}{\varepsilon}\right),
    \end{split}
    \end{equation}
    where $W$ is given in (\ref{gpsj}). Thus the result holds.
    \end{proof}

\section{{\bfseries Calculations of the error term}}\label{sectcet}

    Recall that we are looking for a solution of (\ref{Pb}) in the form
        \begin{equation*}
        u_{\varepsilon}=P_{\varepsilon,\Omega}U_{\varepsilon,\xi_{\varepsilon}}+\omega_{\varepsilon,\xi_{\varepsilon}},
        \end{equation*}
    for suitable $\omega_{\varepsilon,\xi_{\varepsilon}}\in E_{\varepsilon,\xi}$. We need to estimate the error term $J_\varepsilon(P_{\varepsilon,\Omega}U_{\varepsilon,\xi})$. Now, let us give more accurate information of $D_{\varepsilon,R}$, that is
    \begin{equation}\label{edvv}
    D_{\varepsilon,R}=\{\xi\in \Omega: \varepsilon R\leq \varepsilon^{1-\varpi}\leq d(\xi,\partial\Omega)\leq 1/R\},
    \end{equation}
    as $\varepsilon$ is small, where
    \begin{equation}\label{defwj}
    \frac{N-2s}{3(N+2s)}<\varpi<\frac{1}{3},
    \end{equation}
    and in the rest of this paper, we will give more information about $\varpi$.

    \begin{proposition}\label{proejbu}
    If $\varepsilon>0$ is small, we have
    \begin{equation}\label{ejbu}
    \begin{split}
    J_\varepsilon(P_{\varepsilon,\Omega}U_{\varepsilon,\xi})= & \varepsilon^{N}A_1+\frac{1}{4}\varepsilon^{N+2s}A_2
    -\frac{1}{4}\varepsilon^{2N}B^2 H(\xi,\xi)+\frac{1}{2}\tau_{\varepsilon,\xi}  \\
    &+\varepsilon^{N}O\left(\left(\frac{\varepsilon}{d(\xi,\partial\Omega)}\right)^{(1+p)(N+2s)-N}
    +\mu_{\varepsilon,\xi}
    +\varepsilon^{2s+\zeta_0(N-2s)}\right).
    \end{split}
    \end{equation}
    where $\zeta_0\in(\frac{1}{3},1)$ closes to $\frac{1}{3}$, $\mu_{\varepsilon,\xi}$ is given in (\ref{defmu}) and
    \begin{equation}\label{ejbud3}
    \begin{split}
    A_1= & \left(\frac{1}{2}-\frac{1}{p+1}\right)\int_{\mathbb{R}^N}U^{p+1},  \ \
    A_2=\int_{\mathbb{R}^N}U^2 W=c_{N,s}\int_{\mathbb{R}^N}\frac{U^2(y)U^2(x)}{|y-x|^{N-2s}}dydx, \\
    B= & \int_{\mathbb{R}^N}U^2,  \ \
    \tau_{\varepsilon,\xi}=\int_{\Omega}(U_{\varepsilon,\xi})^p\left(U_{\varepsilon,\xi}- P_{\varepsilon,\Omega}U_{\varepsilon,\xi}\right).
    \end{split}
    \end{equation}
    \end{proposition}

    \begin{proof}
    Recalling the definition of $J_\varepsilon$ and using $P_{\varepsilon,\Omega}U_{\varepsilon,\xi}$ as a text function in (\ref{Pp}), we get
    \begin{equation}\label{ejbujb}
    \begin{split}
    J_{\varepsilon}(P_{\varepsilon,\Omega}U_{\varepsilon,\xi})= & \frac{1}{2}\int_{\mathbb{R}^N}\left(\varepsilon^{2s}|(-\Delta)^{\frac{s}{2}}P_{\varepsilon,\Omega}U_{\varepsilon,\xi}|^2+ (P _{\varepsilon,\Omega}U_{\varepsilon,\xi})^2\right)dx \\
    & -\frac{1}{p+1}\int_{\Omega} (P_{\varepsilon,\Omega}U_{\varepsilon,\xi})^{p+1}dx+\frac{1}{4}\int_{\Omega}P^2_{\varepsilon,\Omega}U_{\varepsilon,\xi} \Phi[ P^2_{\varepsilon,\Omega}U_{\varepsilon,\xi}] dx \\
    = & \frac{1}{2}\int_{\Omega}(U_{\varepsilon,\xi})^p P_{\varepsilon,\Omega}U_{\varepsilon,\xi}
    -\frac{1}{p+1}\int_{\Omega} (P_{\varepsilon,\Omega}U_{\varepsilon,\xi})^{p+1}dx  \\
    & +\frac{1}{4}\int_{\Omega}P^2_{\varepsilon,\Omega}U_{\varepsilon,\xi} \Phi[ P^2_{\varepsilon,\Omega}U_{\varepsilon,\xi}] dx  \\
    =: & J^1_{\varepsilon}
    +J^2_{\varepsilon}.
    \end{split}
    \end{equation}
    Firstly, we estimate $J^1_{\varepsilon}$. From $(\ref{upu})$ we know that as $\varepsilon\to 0$,
    $$
    |U_{\varepsilon,\xi}-P_{\varepsilon,\Omega}U_{\varepsilon,\xi}|\leq C \left(\frac{\varepsilon}{d(\xi,\partial\Omega)}\right)^{N+2s}\to 0,
    $$
    thus for any $l>0$~, it holds that
    \begin{equation*}
    \begin{split}
    \int_{\Omega}(U_{\varepsilon,\xi})^{p-l}(U_{\varepsilon,\xi}- P_{\varepsilon,\Omega}U_{\varepsilon,\xi})^{l+1}
    \leq & \left|\frac{U_{\varepsilon,\xi}- P_{\varepsilon,\Omega}U_{\varepsilon,\xi}}{U_{\varepsilon,\xi}}\right|^l_{\infty}\int_{\Omega}(U_{\varepsilon,\xi})^p(U_{\varepsilon,\xi}- P_{\varepsilon,\Omega}U_{\varepsilon,\xi})  \\
    = & o(1)\int_{\Omega}(U_{\varepsilon,\xi})^p(U_{\varepsilon,\xi}- P_{\varepsilon,\Omega}U_{\varepsilon,\xi})
    \end{split}
    \end{equation*}
    By using Taylor's expansion for $(P_{\varepsilon,\Omega}U_{\varepsilon,\xi})^{p+1}$ at $U_{\varepsilon,\xi}$, we have
    \begin{equation}\label{ejbuj1}
    \begin{split}
    J^1_{\varepsilon}
    = & \frac{1}{2}\int_{\Omega}(U_{\varepsilon,\xi})^p\left(U_{\varepsilon,\xi}-(U_{\varepsilon,\xi}-P_{\varepsilon,\Omega}U_{\varepsilon,\xi})\right)
    -\frac{1}{p+1}\int_{\Omega} (U_{\varepsilon,\xi}+(P_{\varepsilon,\Omega}U_{\varepsilon,\xi}-U_{\varepsilon,\xi}))^{p+1} \\
    = & \frac{1}{2}\int_{\Omega}(U_{\varepsilon,\xi})^{p+1}
    -\frac{1}{2}\int_{\Omega}(U_{\varepsilon,\xi})^p(U_{\varepsilon,\xi}-P_{\varepsilon,\Omega}U_{\varepsilon,\xi}) \\
    &-\int_{\Omega}\left(\frac{(U_{\varepsilon,\xi})^{p+1}}{p+1}+(U_{\varepsilon,\xi})^p(P_{\varepsilon,\Omega}U_{\varepsilon,\xi}-U_{\varepsilon,\xi})\right)
    + o(1)\int_{\Omega}(U_{\varepsilon,\xi})^p(U_{\varepsilon,\xi}- P_{\varepsilon,\Omega}U_{\varepsilon,\xi})  \\
    = & \left(\frac{1}{2}-\frac{1}{p+1}\right)\int_{\Omega}(U_{\varepsilon,\xi})^{p+1}
    +\left(\frac{1}{2}+o(1)\right)\int_{\Omega}(U_{\varepsilon,\xi})^p(U_{\varepsilon,\xi}-P_{\varepsilon,\Omega}U_{\varepsilon,\xi})  \\
    = & \varepsilon^N\left(\frac{1}{2}-\frac{1}{p+1}\right)\int_{\Omega_{\varepsilon,\xi}}U^{p+1}
    +\frac{1}{2}\int_{\Omega}(U_{\varepsilon,\xi})^p(U_{\varepsilon,\xi}- P_{\varepsilon,\Omega}U_{\varepsilon,\xi})  \\
    = & \varepsilon^N A+\frac{1}{2}\tau_{\varepsilon,\xi}
    -\varepsilon^N\left(\frac{1}{2}-\frac{1}{p+1}\right)\int_{\mathbb{R}^N\backslash\Omega_{\varepsilon,\xi}}U^{p+1}.
    \end{split}
    \end{equation}
    Noting that $U$ is small in $\mathbb{R}^N\backslash\Omega_{\varepsilon,\xi}$ if $\varepsilon$ is small, then we have
    \begin{equation}\label{ejbuj1p1}
    \begin{split}
    \int_{\mathbb{R}^N\backslash\Omega_{\varepsilon,\xi}}U^{p+1}
    \leq & C\int^{+\infty}_{\frac{d(\xi,\partial\Omega)}{\varepsilon}}\frac{t^{N-1}}{(1+t^{N+2s})^{p+1}}dt \\
    = & O\left(\left(\frac{\varepsilon}{d(\xi,\partial\Omega)}\right)^{(1+p)(N+2s)-N}\right)  \\
    \end{split}
    \end{equation}
    So it remains to estimate $J^2_{\varepsilon}$. We have
    \begin{equation}\label{ejbuj2}
    \begin{split}
    J^2_{\varepsilon}
    = & \frac{1}{4}\int_{\Omega}(U_{\varepsilon,\xi}+(P_{\varepsilon,\Omega}U_{\varepsilon,\xi}-U_{\varepsilon,\xi}))^2 \Phi[(U_{\varepsilon,\xi}+(P_{\varepsilon,\Omega}U_{\varepsilon,\xi}-U_{\varepsilon,\xi}))^2] dx \\
    = & \frac{1}{4}\int_{\Omega}U_{\varepsilon,\xi}^2 \Phi[U_{\varepsilon,\xi}^2] dx \\
    & + \int_{\Omega}
    \Big[
    U_{\varepsilon,\xi}(P_{\varepsilon,\Omega}U_{\varepsilon,\xi}-U_{\varepsilon,\xi}) \Phi[U_{\varepsilon,\xi}^2]
    +U_{\varepsilon,\xi}(P_{\varepsilon,\Omega}U_{\varepsilon,\xi}-U_{\varepsilon,\xi}) \Phi[(P_{\varepsilon,\Omega}U_{\varepsilon,\xi}-U_{\varepsilon,\xi})^2] \\
    & +U_{\varepsilon,\xi}(P_{\varepsilon,\Omega}U_{\varepsilon,\xi}-U_{\varepsilon,\xi}) \Phi[U_{\varepsilon,\xi}(P_{\varepsilon,\Omega}U_{\varepsilon,\xi}-U_{\varepsilon,\xi})]
    +\frac{1}{2}(P_{\varepsilon,\Omega}U_{\varepsilon,\xi}-U_{\varepsilon,\xi})^2 \Phi[U_{\varepsilon,\xi}^2] \\
    & +\frac{1}{4}(P_{\varepsilon,\Omega}U_{\varepsilon,\xi}-U_{\varepsilon,\xi})^2 \Phi[(P_{\varepsilon,\Omega}U_{\varepsilon,\xi}-U_{\varepsilon,\xi})^2]
    \Big] dx \\
    := & J^{21}_{\varepsilon}+J^{22}_{\varepsilon}.
    \end{split}
    \end{equation}
    From Lemma \ref{lemegu}, we obtain that
    \begin{equation}\label{ejbuj21}
    \begin{split}
    4J^{21}_{\varepsilon}
    = & \int_{\Omega}U^2_{\varepsilon,\xi} \Phi[U_{\varepsilon,\xi}^2] dx \\
    = & \varepsilon^{2s}\int_{\Omega}U^2_{\varepsilon,\xi} W\left(\frac{x-\xi}{\varepsilon}\right) dx
    -\varepsilon^{N}B \int_{\Omega}U^2_{\varepsilon,\xi}H(x,\xi) dx
     + \int_{\Omega}U^2_{\varepsilon,\xi} O\left(\mu_{\varepsilon,\xi}\right) dx \\
    = & \varepsilon^{N+2s}\int_{\Omega_{\varepsilon,\xi}}U^2 W dz
    -\varepsilon^{2N}B H(\xi,\xi) \int_{\Omega_{\varepsilon,\xi}}U^2dz
     + \varepsilon^{N}O\left(\mu_{\varepsilon,\xi}\right)\int_{\Omega_{\varepsilon,\xi}}U^2 dz \\
    = & \varepsilon^{N+2s}\int_{\mathbb{R}^N}U^2 W dz
    -\varepsilon^{2N}B H(\xi,\xi) \int_{\mathbb{R}^N}U^2 dz  + O\left(\frac{\varepsilon^{3N+4s}}{d(\xi,\partial\Omega)^{2N+2s}}\right)\\
    & + \varepsilon^{N}O\left(\mu_{\varepsilon,\xi}\right)\int_{\Omega_{\varepsilon,\xi}}U^2 dz  \\
    = & \varepsilon^{N+2s}A_2
    -\varepsilon^{2N}B^2 H(\xi,\xi)
    + \varepsilon^{N}O\left(\mu_{\varepsilon,\xi}\right).
    \end{split}
    \end{equation}
    On the other hand, since $P_{\varepsilon,\Omega}U_{\varepsilon,\xi}\leq U_{\varepsilon,\xi}$ and by the behavior of $\Phi[U_{\varepsilon,\xi}^2]$ as shown in Lemma \ref{lemeguc}, we have 
    \begin{equation*}
    \begin{split}
    |J^{22}_{\varepsilon}|
    \leq & C\int_{\Omega}U_{\varepsilon,\xi}|P_{\varepsilon,\Omega}U_{\varepsilon,\xi}-U_{\varepsilon,\xi}| \Phi[U_{\varepsilon,\xi}^2]dx \\
    \leq & C|P_{\varepsilon,\Omega}U_{\varepsilon,\xi}-U_{\varepsilon,\xi}|_{\infty}\int_{\Omega}U_{\varepsilon,\xi} \Phi[U_{\varepsilon,\xi}^2] dx  \\
    \leq & C \left(\frac{\varepsilon}{d(\xi,\partial\Omega)}\right)^{N+2s}\varepsilon^{2s} \varepsilon^{N} \int_{\Omega_{\varepsilon,\xi}}\frac{1}{(1+|x|^{N-2s})(1+|x|^{N+2s})} dx  \\
    \leq & C\varepsilon^{N+2s+\varpi(N+2s)}.
    \end{split}
    \end{equation*}
    From (\ref{defwj}), we deduce that
    \begin{equation*}
    N+2s+\varpi (N+2s)>\frac{4N+4s}{3},
    \end{equation*}
    then we can choose $\zeta_0\in(\frac{1}{3},1)$ closes to $\frac{1}{3}$ which is same as in Lemma \ref{lemel}, such that
    \begin{equation}\label{ejbuj22}
    \begin{split}
    |J^{22}_{\varepsilon}|= & \varepsilon^{N}o\left(\varepsilon^{2s+\zeta_0(N-2s)}\right).
    \end{split}
    \end{equation}
    Thus the result holds.
    \end{proof}

    Now, let us estimate $\tau_{\varepsilon,\xi}$ which is defined in (\ref{ejbud3}). In order to obtain a proper result, we will introduce a new ``nonlocal normal derivative" which is simply and helpful for this estimation. It is well known the Green's formula for classical Laplacian $\Delta$ that
    \begin{equation*}
    \int_{\Omega}\nabla u\cdot \nabla v=\int_{\Omega}v(-\Delta)u+\int_{\partial\Omega}v\partial_{\nu}u.
    \end{equation*}
    In \cite{drv}, Dipierro, Ros-oton and Valdinoci have established a new ``nonlocal normal derivative" $\mathcal{N}_s$ which is defined as
    \begin{equation}\label{defpfl}
    \mathcal{N}_s u(x):=C_{N,s} \int_{\Omega}\frac{u(x)-u(y)}{|x-y|^{N+2s}}dy,\ \ x\in \mathbb{R}^N\backslash\overline{\Omega},
    \end{equation}
    and they get the Green's formula corresponding to fractional Laplacian $(-\Delta)^s$,
    \begin{equation}\label{defgffl}
    \frac{C_{N,s}}{2}\int_{\mathcal{D}_{\Omega}}\frac{(u(x)-u(y))(v(x)-v(y))}{|x-y|^{N+2s}} dxdy=\int_{\Omega}v(-\Delta)^s u+\int_{\mathbb{R}^N\backslash\Omega}v\mathcal{N}_s u.
    \end{equation}

    Following the proof of \cite[Lemma 2.1]{dy2}, we have:

    \begin{lemma}\label{lemetx}
    If $\varepsilon$ is small, then there are $0<c_1\leq c_2$ such that
    \begin{equation}\label{tx}
    \begin{split}
    c_1\varepsilon^N\left(\frac{\varepsilon}{d(\xi,\partial\Omega)}\right)^{N+4s}
    \leq\tau_{\varepsilon,\xi}
    \leq c_2\varepsilon^N\left(\frac{\varepsilon}{d(\xi,\partial\Omega)}\right)^{N+4s}.
    \end{split}
    \end{equation}
    \end{lemma}

    \begin{proof}
    Multiplying (\ref{defetae}) by $U_{\varepsilon,\xi}$, $\eta_{\varepsilon,\xi}$ respectively and integrating, we get
    \begin{equation}\label{txu}
    \begin{split}
    \varepsilon^{2s}\int_{\Omega}U_{\varepsilon,\xi}(-\Delta)^s \eta_{\varepsilon,\xi}+ \int_{\Omega}U_{\varepsilon,\xi}\eta_{\varepsilon,\xi}=0,
    \end{split}
    \end{equation}
    \begin{equation}\label{txeta}
    \begin{split}
    \varepsilon^{2s}\int_{\Omega}\eta_{\varepsilon,\xi}(-\Delta)^s \eta_{\varepsilon,\xi}+ \int_{\Omega}\eta_{\varepsilon,\xi}\eta_{\varepsilon,\xi}=0.
    \end{split}
    \end{equation}
    Then multiplying (\ref{Pp}) by $\eta_{\varepsilon,\xi}$ and integrating by parts, we obtain
    \begin{equation}\label{txe}
    \begin{split}
    \tau_{\varepsilon,\xi} = & \int_{\Omega}(U_{\varepsilon,\xi})^p(U_{\varepsilon,\xi}-P_{\varepsilon,\xi}U_{\varepsilon,\xi})  \\
    = & \varepsilon^{2s}\int_{\Omega}\eta_{\varepsilon,\xi}(-\Delta)^s P_{\varepsilon,\Omega}U_{\varepsilon,\xi}
    +\int_{\Omega}\eta_{\varepsilon,\xi}P_{\varepsilon,\Omega}U_{\varepsilon,\xi}  \\
    = & \varepsilon^{2s}\int_{\Omega}\eta_{\varepsilon,\xi}(-\Delta)^s (U_{\varepsilon,\xi}-\eta_{\varepsilon,\xi})
    +\int_{\Omega}\eta_{\varepsilon,\xi}(U_{\varepsilon,\xi}-\eta_{\varepsilon,\xi})  \\
    = & \varepsilon^{2s}\int_{\Omega}\eta_{\varepsilon,\xi}(-\Delta)^s U_{\varepsilon,\xi}
    +\int_{\Omega}\eta_{\varepsilon,\xi}U_{\varepsilon,\xi}  -\left( \varepsilon^{2s}\int_{\Omega}\eta_{\varepsilon,\xi}(-\Delta)^s \eta_{\varepsilon,\xi}+\int_{\Omega}\eta_{\varepsilon,\xi}\eta_{\varepsilon,\xi}\right)  \\
    = & \varepsilon^{2s}\int_{\Omega}\eta_{\varepsilon,\xi}(-\Delta)^s U_{\varepsilon,\xi}
    -\varepsilon^{2s}\int_{\Omega}U_{\varepsilon,\xi}(-\Delta)^s \eta_{\varepsilon,\xi}  \\
    = & \varepsilon^{2s}\int_{\mathbb{R}^N\backslash\Omega}U_{\varepsilon,\xi}\mathcal{N}_s \eta_{\varepsilon,\xi}
    -\varepsilon^{2s}\int_{\mathbb{R}^N\backslash\Omega}\eta_{\varepsilon,\xi}\mathcal{N}_s U_{\varepsilon,\xi}  \\
    = & \varepsilon^{2s}\int_{\mathbb{R}^N\backslash\Omega}U_{\varepsilon,\xi}\mathcal{N}_s \eta_{\varepsilon,\xi}
    -\varepsilon^{2s}\int_{\mathbb{R}^N\backslash\Omega}U_{\varepsilon,\xi}\mathcal{N}_s U_{\varepsilon,\xi}.
    \end{split}
    \end{equation}
    Besides, using the Green's formula (\ref{defgffl}) for (\ref{txeta}), we obtain
    \begin{equation}\label{txea}
    \begin{split}
    \varepsilon^{2s}\int_{\mathbb{R}^N\backslash\Omega}U_{\varepsilon,\xi}\mathcal{N}_s \eta_{\varepsilon,\xi}
    =\varepsilon^{2s}\int_{\mathbb{R}^N\backslash\Omega}\eta_{\varepsilon,\xi}\mathcal{N}_s \eta_{\varepsilon,\xi}
    =\|\eta_{\varepsilon,\xi}\|^2_{\varepsilon}>0.
    \end{split}
    \end{equation}
    Combining (\ref{txe}) and (\ref{txea}), we obtain
    \begin{equation}\label{txed}
    \begin{split}
    \tau_{\varepsilon,\xi}> & -\varepsilon^{2s}\int_{\mathbb{R}^N\backslash\Omega}U_{\varepsilon,\xi}\mathcal{N}_s U_{\varepsilon,\xi}  \\
    = & -C_{N,s} \varepsilon^{2s}\int_{\mathbb{R}^N\backslash\Omega}U_{\varepsilon,\xi}(x)\int_{\Omega}\frac{U_{\varepsilon,\xi}(x)-U_{\varepsilon,\xi}(y)}{|x-y|^{N+2s}}dydx \\
    = & -C_{N,s} \varepsilon^{N}\int_{\mathbb{R}^N\backslash\Omega_{\varepsilon,\xi}}U(x)\int_{\Omega_{\varepsilon,\xi}}\frac{U(x)-U(y)}{|x-y|^{N+2s}}dydx \\
    = & C_{N,s} \varepsilon^{N}\int_{\mathbb{R}^N\backslash\Omega_{\varepsilon,\xi}}U(x)\int_{\Omega_{\varepsilon,\xi}}\frac{U(y)-U(x)}{|y-x|^{N+2s}}dydx.
    \end{split}
    \end{equation}
    We can choose small $\varrho>0$ which is independent of $\varepsilon$ such that there exists $c>0$ such that $U(y)-U(x)\geq c U(y)$, $\forall x\in \mathbb{R}^N\backslash \Omega_{\varepsilon,\xi},\ \forall y\in B_{\frac{\varrho d(\xi,\partial\Omega)}{\varepsilon}}(0)$.
    Define
    \begin{equation*}
    \mathcal{S}=\{x\in \mathbb{R}^N\backslash\Omega_{\varepsilon,\xi}: \varrho d(\xi,\partial\Omega)/\varepsilon\leq d(x,\partial\Omega_{\varepsilon,\xi})\leq 3\varrho d(\xi,\partial\Omega)/\varepsilon\}.
    \end{equation*}
    Then we can find a suitable $\xi^1\in\mathcal{S}$ such that $B_{\frac{\varrho d(\xi,\partial\Omega)}{\varepsilon}}(\xi^1)\subset \mathcal{S}$ and for all $x\in B_{\frac{\varrho d(\xi,\partial\Omega)}{\varepsilon}}(\xi^1)$, $(1+\varrho)d(\xi,\partial\Omega)/\varepsilon\leq |x| \leq (1+3\varrho)d(\xi,\partial\Omega)/\varepsilon$.
    Then $\forall x\in B_{\frac{\varrho d(\xi,\partial\Omega)}{\varepsilon}}(\xi^1),\ \forall y\in B_{\frac{\varrho d(\xi,\partial\Omega)}{\varepsilon}}(0)$, it holds that
    \begin{equation*}
    d(\xi,\partial\Omega)/\varepsilon\leq |x-y|\leq (1+4\varrho)d(\xi,\partial\Omega)/\varepsilon.
    \end{equation*}
    Noticing the fact that
    \begin{equation*}
    B_{R}(0)\subset B_{\frac{\varrho d(\xi,\partial\Omega)}{\varepsilon}}(0)\subset \Omega_{\varepsilon,\xi}
    \end{equation*}
    for large $R>0$, and by (\ref{bu}), we obtain
    \begin{equation}\label{txeam}
    \begin{split}
    \tau_{\varepsilon,\xi}
    \geq & C_1 \varepsilon^{N}\int_{B_{\frac{\varrho  d(\xi,\partial\Omega)}{\varepsilon}}(\xi^1)}U(x)\int_{B_{\frac{\varrho  d(\xi,\partial\Omega)}{\varepsilon}}(0)}\frac{U(y)-U(x)}{|y-x|^{N+2s}}dydx \\
    \geq & C_2 \varepsilon^{N}\int_{B_{\frac{\varrho  d(\xi,\partial\Omega)}{\varepsilon}}(\xi^1)}U(x)\int_{B_{\frac{\varrho  d(\xi,\partial\Omega)}{\varepsilon}}(0)}\frac{U(y)}{|y-x|^{N+2s}}dydx \\
    \geq & C_3\varepsilon^{N}\left(\frac{\varepsilon}{d(\xi,\partial\Omega)}\right)^{N+2s}\left(\frac{d(\xi,\partial\Omega)}{\varepsilon}\right)^{N}
    \left(\frac{\varepsilon}{d(\xi,\partial\Omega)}\right)^{N+2s}
    \int_{B_{\frac{\varrho  d(\xi,\partial\Omega)}{\varepsilon}}(0)}U(y)dy  \\
    \geq & C_4
    \left(\frac{\varepsilon}{d(\xi,\partial\Omega)}\right)^{N+4s}\int_{B_R(0)}U(y)dy  \\
    \geq & C_5\varepsilon^{N}\left(\frac{\varepsilon}{d(\xi,\partial\Omega)}\right)^{N+4s}.
    \end{split}
    \end{equation}

    On the other hand, from (\ref{txe}), we have
    \begin{equation}\label{txex}
    \begin{split}
    \tau_{\varepsilon,\xi}= & \varepsilon^{2s}\int_{\mathbb{R}^N\backslash\Omega}U_{\varepsilon,\xi}\left(\mathcal{N}_s \eta_{\varepsilon,\xi}
    -\mathcal{N}_s U_{\varepsilon,\xi}\right)  \\
    = & C_{N,s}\varepsilon^{2s}\int_{\mathbb{R}^N\backslash\Omega}U_{\varepsilon,\xi}(x)\int_{\Omega}
    \frac{P_{\varepsilon,\Omega}U_{\varepsilon,\xi}(y)-P_{\varepsilon,\Omega}U_{\varepsilon,\xi}(x)}{|y-x|^{N+2s}}dydx  \\
    = & C_{N,s}\varepsilon^{2s}\int_{\mathbb{R}^N\backslash\Omega}U_{\varepsilon,\xi}(x)\int_{\Omega}
    \frac{P_{\varepsilon,\Omega}U_{\varepsilon,\xi}(y)}{|y-x|^{N+2s}}dydx  \\
    & -C_{N,s}\varepsilon^{2s}\int_{\mathbb{R}^N\backslash\Omega}U_{\varepsilon,\xi}(x)P_{\varepsilon,\Omega}U_{\varepsilon,\xi}(x)\int_{\Omega}
    \frac{1}{|y-x|^{N+2s}}dydx  \\
    = & C_{N,s}\varepsilon^{2s}\int_{\mathbb{R}^N\backslash\Omega}U_{\varepsilon,\xi}(x)\int_{\Omega}
    \frac{P_{\varepsilon,\Omega}U_{\varepsilon,\xi}(y)}{|y-x|^{N+2s}}dydx \\
    = & C_{N,s}\varepsilon^{N}\int_{\mathbb{R}^N\backslash\Omega_{\varepsilon,\xi}}U(x)\int_{\Omega_{\varepsilon,\xi}}\frac{P_{\varepsilon,\Omega_{\varepsilon,\xi}}U(y)}{|y-x|^{N+2s}}dydx.
    \end{split}
    \end{equation}
    If $x\in \mathbb{R}^N\backslash\Omega_{\varepsilon,\xi}$, then $\frac{1}{|x|^{N+2s}}\leq \frac{2}{1+|x|^{N+2s}}$ since $|x|>>1$. Then from
    \cite[Lemma 5.1]{dddv}, we have
    \begin{equation}\label{txexxhbg}
    \begin{split}
    \tau_{\varepsilon,\xi}\leq
    C\varepsilon^{N}\int_{\mathbb{R}^N\backslash\Omega_{\varepsilon,\xi}}\frac{1}{(1+|x|^{N+2s})^2}dx
    \leq C'\varepsilon^{N}\left(\frac{\varepsilon}{d(\xi,\partial\Omega)}\right)^{N+4s}.
    \end{split}
    \end{equation}
    In fact, from (\ref{defetae}), we know that $\eta_{\varepsilon,\xi}\geq 0$ in $\mathbb{R}^N$, i.e. $P_{\varepsilon,\Omega}U_{\varepsilon,\xi}\leq
    U_{\varepsilon,\xi}$.
    Therefore,
    \begin{equation}\label{txexx}
    \begin{split}
    \tau_{\varepsilon,\xi}\leq & C_{N,s}\varepsilon^{2s}\int_{\mathbb{R}^N\backslash\Omega}U_{\varepsilon,\xi}(x)\int_{\Omega}
    \frac{U_{\varepsilon,\xi}(y)}{|y-x|^{N+2s}}dydx  \\
    = & C_{N,s}\varepsilon^{N}\int_{\mathbb{R}^N\backslash\Omega_{\varepsilon,\xi}}U(x)\int_{\Omega_{\varepsilon,\xi}}\frac{U(y)}{|y-x|^{N+2s}}dydx.
    \end{split}
    \end{equation}
    Moreover,
    \begin{equation*}
    \begin{split}
    \int_{\Omega_{\varepsilon,\xi}}\frac{U(y)}{|y-x|^{N+2s}}dy
    \leq & \int_{\Omega_{\varepsilon,\xi}}\frac{U(y)}{||x|^{N+2s}-|y|^{N+2s}|}dy  \\
    \leq & C_6\int_{\Omega_{\varepsilon,\xi}}\frac{1}{(1+|y|^{N+2s})(||x|^{N+2s}-|y|^{N+2s}|)}dy  \\
    \leq & \frac{C_6}{1+|x|^{N+2s}}\int_{\Omega_{\varepsilon,\xi}}
    \left(\frac{1}{1+|y|^{N+2s}}+\frac{1}{||x|^{N+2s}-|y|^{N+2s}|}\right)dy,
    \end{split}
    \end{equation*}
    then
    \begin{equation*}
    \begin{split}
    \tau_{\varepsilon,\xi}\leq &
    C_6\varepsilon^{N}\int_{\mathbb{R}^N\backslash\Omega_{\varepsilon,\xi}}\frac{1}{(1+|x|^{N+2s})^2}dx
    \int_{\Omega_{\varepsilon,\xi}}\frac{1}{1+|y|^{N+2s}}dy  \\
    & +C_6\varepsilon^{N}\int_{\mathbb{R}^N\backslash\Omega_{\varepsilon,\xi}}\frac{1}{(1+|x|^{N+2s})^2}
    \int_{\Omega_{\varepsilon,\xi}}\frac{1}{||x|^{N+2s}-|y|^{N+2s}|}dydx \\
    \leq & C_7\varepsilon^{N}\left(\frac{\varepsilon}{d(\xi,\partial\Omega)}\right)^{N+4s}  \\
    & +C_6\varepsilon^{N}\int_{\mathbb{R}^N\backslash\Omega_{\varepsilon,\xi}}\frac{1}{(1+|x|^{N+2s})^2}
    \int_{\Omega_{\varepsilon,\xi}}\frac{1}{||x|^{N+2s}-|y|^{N+2s}|}dydx.
    \end{split}
    \end{equation*}
    What's more,
    \begin{equation}\label{txexxw}
    \begin{split}
    & \int_{\mathbb{R}^N\backslash\Omega_{\varepsilon,\xi}}\frac{1}{(1+|x|^{N+2s})^2}
    \int_{\Omega_{\varepsilon,\xi}}\frac{1}{||x|^{N+2s}-|y|^{N+2s}|}dydx  \\
    \leq & \int_{\mathbb{R}^N\backslash\Omega_{\varepsilon,\xi}}\frac{1}{(1+|x|^{N+2s})^2}
    \int_{\Omega_{\varepsilon,\xi}}\frac{1}{||x|^N-|y|^N|(|x|^{2s}+|y|^{2s})}dydx  \\
    \leq & \int_{\mathbb{R}^N\backslash\Omega_{\varepsilon,\xi}}\frac{1}{(1+|x|^{N+2s})^2}\frac{1}{|x|^{2s}}
    \int_{\Omega_{\varepsilon,\xi}}\frac{1}{||x|^N-|y|^N|}dydx \\
    = & o\left(\left(\frac{\varepsilon}{d(\xi,\partial\Omega)}\right)^{N+4s}\right),
    \end{split}
    \end{equation}
    thus we have
    \begin{equation}\label{txexxhb}
    \begin{split}
    \tau_{\varepsilon,\xi}
    \leq & C_8\varepsilon^{N}\left(\frac{\varepsilon}{d(\xi,\partial\Omega)}\right)^{N+4s}.
    \end{split}
    \end{equation}
    Thus the result follows by (\ref{txeam}) and (\ref{txexxhb}).
    \end{proof}

    \begin{remark}\label{remtxexxw}\rm
    We remark that the integral
    $$
    \int_{\mathbb{R}^N\backslash B_R(0)}\frac{1}{(1+|x|^{N+2s})^2}\frac{1}{|x|^{2s}}\int_{B_R(0)}\frac{1}{||x|^N-|y|^N|}dydx
    $$
    is integrable for each $R>0$, moreover, this integral value is decreasing for large $R>0$. Indeed,
    \begin{equation*}
    \begin{split}
    & \int_{\mathbb{R}^N\backslash B_R(0)}\frac{1}{(1+|x|^{N+2s})^2}\frac{1}{|x|^{2s}}\int_{B_R(0)}\frac{1}{||x|^N-|y|^N|}dydx  \\
    = & C \int^{+\infty}_{R}\frac{t^{N-1-2s}}{(1+t^{N+2s})^2}\int^{R}_0\frac{r^{N-1}}{t^N-r^N}drdt  \\
    = & C \int^{+\infty}_{R}\frac{t^{N-1-2s}\log (\frac{t^N}{t^N-R^N})}{(1+t^{N+2s})^2}dt  \\
    = & C' \int^{+\infty}_{1}\frac{(\log r)R^{N-2s}}
    {(r-1)^{2-\frac{2s}{N}}r^{\frac{2s}{N}}\left[1+(\frac{r}{r-1})^{\frac{N+2s}{N}}R^{N+2s}\right]^2}dr \\
    =  & C' \int^{+\infty}_{1}\frac{(\log r) (\frac{r-1}{r})^{\frac{N+6s}{N}}}
    {(r-1)^{\frac{2s}{N}}\left[(r-1)^{1-\frac{2s}{N}}r\right]\left[(\frac{r-1}{r})^{\frac{N+2s}{N}}R^{-\frac{N-2s}{2}}+R^{\frac{N+6s}{2}}\right]^2}dr \\
    \end{split}
    \end{equation*}
    since $\log r\leq (r-1)^{\frac{2s}{N}}$ for all $r\geq 1$ and $N>2s$, we can know the above integral is integrable for each $R>0$ and the integral value is decreasing for large $R>0$. Therefore, we have
    \begin{equation*}
    \begin{split}
    &\int_{\mathbb{R}^N\backslash\Omega_{\varepsilon,\xi}}\frac{1}{(1+|x|^{N+2s})^2}\frac{1}{|x|^{2s}}
    \int_{\Omega_{\varepsilon,\xi}}\frac{1}{||x|^N-|y|^N|}dydx  \\
    \leq & \int_{\mathbb{R}^N\backslash B_{\frac{d(\xi,\partial\Omega)}{\varepsilon}}(0)}\frac{1}{(1+|x|^{N+2s})^2}\frac{1}{|x|^{2s}}
    \int_{B_{\frac{d(\xi,\partial\Omega)}{\varepsilon}}(0)}\frac{1}{||x|^N-|y|^N|}dydx  \\
    \leq & C \left(\frac{\varepsilon}{d(\xi,\partial\Omega)}\right)^{N+6s}
    \int^{+\infty}_{1}\frac{1}{(r-1)^{1-\frac{2s}{N}}r}dr   \\
    \leq & C' \left(\frac{\varepsilon}{d(\xi,\partial\Omega)}\right)^{N+6s}=o\left(\left(\frac{\varepsilon}{d(\xi,\partial\Omega)}\right)^{N+4s}\right).
    \end{split}
    \end{equation*}
    Thus (\ref{txexxw}) holds.
    \end{remark}

\section{{\bfseries The reduction}}\label{secttl}

    In this section, we will calculate $L_{\varepsilon,\xi}$, $Q_{\varepsilon,\xi}$, $R_{\varepsilon}$  and develop the reduction. Here, we take $\varpi$ for $D_{\varepsilon,R}$ in (\ref{edvv}) as
    \begin{equation}\label{defw}
    \max\left\{\frac{N+4s}{6[p(N+2s)-N/2]},\frac{N+4s}{6(N+2s)}\right\}<\varpi<\frac{1}{3},
    \end{equation}
    in which $N<8s$.
    Firstly, we prove the following lemma.

    \begin{lemma}\label{lemel}
    As defined in {\rm (\ref{rl})}, we have $L_{\varepsilon,\xi}$ is a bounded linear operator from $E_{\varepsilon,\xi}$ into $\mathbb{R}$. Moreover,
    for all $\omega\in E_{\varepsilon,\xi}$, we have
    \begin{eqnarray}\label{ellw}
    L_{\varepsilon,\xi}(\omega)=O\left(\varepsilon^{\frac{(N+2s)+\zeta_0(N-2s)}{2}}\right)\|\omega\|_\varepsilon,
    \end{eqnarray}
    for some $\zeta_0\in \left(\frac{1}{3},\frac{2s}{N-2s}\right)$ closes to $\frac{1}{3}$.
    In particular there is $l_{\varepsilon,\xi}\in E_{\varepsilon,\xi}$ such that
    \begin{equation}\label{ellwl}
    L_{\varepsilon,\xi}(\omega)=\langle l_{\varepsilon,\xi}, \omega\rangle_{\varepsilon},\ \ \forall \omega\in E_{\varepsilon,\xi}.
    \end{equation}
    \end{lemma}

    \begin{proof}
    Recalling the definition of $L_{\varepsilon,\xi}(\omega)$ and using $\omega$ as a test function in (\ref{Pp}), we get
    \begin{equation*}
    \begin{split}
    L_{\varepsilon,\xi}(\omega)= & \langle P_{\varepsilon,\Omega}U_{\varepsilon,\xi}, w \rangle_\varepsilon
    \int_{\Omega}\Phi[P^2_{\varepsilon,\Omega}U_{\varepsilon,\xi}]P_{\varepsilon,\Omega}U_{\varepsilon,\xi} \omega  dx
    - \int_{\Omega}(P_{\varepsilon,\Omega}U_{\varepsilon,\xi})^p\omega dx  \\
    = & \int_{\Omega}((U_{\varepsilon,\xi})^p-(P_{\varepsilon,\Omega}U_{\varepsilon,\xi})^p)\omega dx
    +\int_{\Omega}\Phi[P^2_{\varepsilon,\Omega}U_{\varepsilon,\xi}]P_{\varepsilon,\Omega}U_{\varepsilon,\xi} \omega  dx  \\
    = & L_{\varepsilon,\xi}^1+L_{\varepsilon,\xi}^2.
    \end{split}
    \end{equation*}
    where
    \begin{equation*}
    \begin{split}
    L_{\varepsilon,\xi}^1=\int_{\Omega}((U_{\varepsilon,\xi})^p-(P_{\varepsilon,\Omega}U_{\varepsilon,\xi})^p)\omega dx,\ \
    L_{\varepsilon,\xi}^2=\int_{\Omega}\Phi[P^2_{\varepsilon,\Omega}U_{\varepsilon,\xi}]P_{\varepsilon,\Omega}U_{\varepsilon,\xi} \omega  dx.
    \end{split}
    \end{equation*}
    Then, using Taylor's expansion for $(U_{\varepsilon,\xi})^p$ at $P_{\varepsilon,\Omega}U_{\varepsilon,\xi}$ and from (\ref{upu}), we get
    \begin{equation*}
    \begin{split}
    L_{\varepsilon,\xi}^1 & \leq C\int_{\Omega}(U_{\varepsilon,\xi})^{p-1}(U_{\varepsilon,\xi}-P_{\varepsilon,\Omega}U_{\varepsilon,\xi})\omega dx  \\
    \leq & C'\left(\frac{\varepsilon}{d(\xi,\partial\Omega)}\right)^{N+2s}
    \left(\varepsilon^N \int_{\Omega_{\varepsilon,\xi}}\left(\frac{1}{1+|x|^{N+2s}}\right)^{2(p-1)} dx\right)^{\frac{1}{2}}\left(\int_{\Omega}|\omega|^2 dx\right)^{\frac{1}{2}}  \\
    \leq & C''\varepsilon^{\frac{N}{2}}\left(\frac{\varepsilon}{d(\xi,\partial\Omega)}\right)^{N+2s}\|\omega\|_\varepsilon \left(\int^{\frac{r}{\varepsilon}}_{0}\frac{t^{N-1}}{\left(1+t^{N+2s}\right)^{2(p-1)}} dx\right)^{\frac{1}{2}} \\
    \end{split}
    \end{equation*}
    where $r=\max_{x,y\in\Omega}|x-y|$. By a simple calculation, it holds that
    \begin{eqnarray*}
    \int^{\frac{r}{\varepsilon}}_{0}\frac{t^{N-1}}{\left(1+t^{N+2s}\right)^{2(p-1)}} dx=
    \left\{ \arraycolsep=1.5pt
       \begin{array}{ll}
        O(\varepsilon^{2(p-1)(N+2s)-N}),\ \ &{\rm if}\ \ 1<p<1+\frac{N}{2(N+2s)},\\[3mm]
        O(|\ln\varepsilon|),\ \ &{\rm if}\ \ p=1+\frac{N}{2(N+2s)},\\[3mm]
        O(1),\ \ &{\rm if}\ \  1+\frac{N}{2(N+2s)}<p<\frac{N+2s}{N-2s}.
        \end{array}
    \right.
    \end{eqnarray*}
    Moreover, if $1<p<1+\frac{N}{2(N+2s)}$, taking $q=1+\frac{N}{2(N+2s)}-p$ then
    \begin{equation*}
    \begin{split}
    L_{\varepsilon,\xi}^1 & \leq C|U_{\varepsilon,\xi}-P_{\varepsilon,\Omega}U_{\varepsilon,\xi}|_\infty^{1-q}\int_{\Omega}(U_{\varepsilon,\xi})^{p-1}(U_{\varepsilon,\xi}-P_{\varepsilon,\Omega}U_{\varepsilon,\xi})^q\omega dx  \\
    \leq & C'\left(\frac{\varepsilon}{d(\xi,\partial\Omega)}\right)^{(1-q)(N+2s)}
    \left(\varepsilon^N \int_{\Omega_{\varepsilon,\xi}}\left(\frac{1}{1+|x|^{N+2s}}\right)^{2(p-1)+2q} dx\right)^{\frac{1}{2}}\left(\int_{\Omega}|\omega|^2 dx\right)^{\frac{1}{2}}  \\
    \leq & C''\varepsilon^{\frac{N}{2}}\left(\frac{\varepsilon}{d(\xi,\partial\Omega)}\right)^{(1-q)(N+2s)}\|\omega\|_\varepsilon \left(\int^{\frac{r}{\varepsilon}}_{0}\frac{t^{N-1}}{1+t^{N}} dx\right)^{\frac{1}{2}} \\
    \leq & C'''\varepsilon^{\frac{N}{2}}\left(\frac{\varepsilon}{d(\xi,\partial\Omega)}\right)^{p(N+2s)-N/2}|\ln\varepsilon|^{\frac{1}{2}}
    \|\omega\|_\varepsilon,
    \end{split}
    \end{equation*}
    then we obtain
    \begin{eqnarray}\label{l1b12}
    L_{\varepsilon,\xi}^1=
    \left\{ \arraycolsep=1.5pt
       \begin{array}{ll}
        O\left(\varepsilon^{\frac{N}{2}+\varpi[p(N+2s)-N/2]}|\ln\varepsilon|^{\frac{1}{2}}\right)\|\omega\|_\varepsilon,\ \ &{\rm if}\ \ 1<p<1+\frac{N}{2(N+2s)};\\[3mm]
        O\left(\varepsilon^{\frac{N}{2}+\varpi(N+2s)}|\ln\varepsilon|^{\frac{1}{2}}\right)\|\omega\|_\varepsilon,\ \ &{\rm if}\ \ p=1+\frac{N}{2(N+2s)};\\[3mm]
        O\left(\varepsilon^{\frac{N}{2}+\varpi(N+2s)}\right)\|\omega\|_\varepsilon,\ \ &{\rm if}\ \  1+\frac{N}{2(N+2s)}<p<\frac{N+2s}{N-2s}.
        \end{array}
    \right.
    \end{eqnarray}
    From (\ref{defw}), we deduce that
    \begin{equation*}
    \frac{N}{2}+\varpi[p(N+2s)-N/2]>\frac{2N+2s}{3}\ \ \mbox{and}\ \ \frac{N}{2}+\varpi(N+2s)>\frac{2N+2s}{3},
    \end{equation*}
    then we can choose $\zeta_0\in(\frac{1}{3},1)$ closes to $\frac{1}{3}$ such that
    \begin{equation}\label{l1bf}
    \begin{split}
    L_{\varepsilon,\xi}^1=O\left(\varepsilon^{\frac{(N+2s)+\zeta_0(N-2s)}{2}}\right)\|\omega\|_\varepsilon.
    \end{split}
    \end{equation}

    Then, we deal with $L_{\varepsilon,\xi}^2$. From (\ref{defetae}), we can know that $P_{\varepsilon,\Omega}U_{\varepsilon,\xi}\leq U_{\varepsilon,\xi}$, and then by the comparison theorem, we get $\Phi[P^2_{\varepsilon,\Omega}U_{\varepsilon,\xi}]\leq \Phi[U^2_{\varepsilon,\xi}]$, thus
    \begin{equation*}
    \begin{split}
    L_{\varepsilon,\xi}^2\leq \int_{\Omega}\Phi[U^2_{\varepsilon,\xi}]U_{\varepsilon,\xi} |\omega|  dx.
    \end{split}
    \end{equation*}
    As a result, from Lemma \ref{lemeguc} and by H\"{o}lder's inequality, we have
    \begin{equation*}
    \begin{split}
    L_{\varepsilon,\xi}^2\leq & C\left(\int_{\Omega}\left(\Phi[U^2_{\varepsilon,\xi}]U_{\varepsilon,\xi}\right)^{2} dx\right)^{\frac{1}{2}}\left(\int_{\Omega}|\omega|^2 dx\right)^{\frac{1}{2}} \\
    \leq & C'\varepsilon^{2s}\left(\varepsilon^N \int_{\Omega_{\varepsilon,\xi}}\frac{1}{1+|x|^{4N}}dx\right)^{\frac{1}{2}}\left(\int_{\Omega}|\omega|^2 dx\right)^{\frac{1}{2}}  \\
    = & O\left(\varepsilon^{\frac{N}{2}+2s}\right)\|\omega\|_\varepsilon
    \end{split}
    \end{equation*}
    Since
    \begin{equation*}
    \begin{split}
    \frac{N}{2}+2s> & \frac{2N+2s}{3}\Leftrightarrow\ N<8s,
    \end{split}
    \end{equation*}
    then we have
    \begin{equation}\label{l2d}
    \begin{split}
    L_{\varepsilon,\xi}^2=& O\left(\varepsilon^{\frac{(N+2s)+\zeta_0(N-2s)}{2}}\right)\|\omega\|_\varepsilon,
    \end{split}
    \end{equation}
    for any $\zeta_0\in \left(\frac{1}{3},\frac{2s}{N-2s}\right)$, if we take $2s<N<8s$.
    Then the result follows by (\ref{l1bf}), (\ref{l2d}).
    \end{proof}

    \begin{lemma}\label{lemqb}
    As the definition of $Q_{\varepsilon,\xi}$ in {\rm (\ref{rq})}, there exists a positive constant $C$, independent of $\varepsilon$, such that
    \begin{equation}\label{qbw}
    Q_{\varepsilon,\xi}(\omega,\eta)\leq C \|\omega\|_{\varepsilon}\|\eta\|_{\varepsilon}.
    \end{equation}
    In particular there is $\mathcal{A}_{\varepsilon,\xi}\in E_{\varepsilon,\xi}$ such that
    \begin{equation}\label{qba}
    Q_{\varepsilon,\xi}(\omega,\eta)=\langle \mathcal{A}_{\varepsilon,\xi}\omega,\eta \rangle_{\varepsilon},\ \ \forall \omega,\ \eta\in E_{\varepsilon,\xi}.
    \end{equation}
    \end{lemma}

    \begin{proof}
    Recalling the definition of $Q_{\varepsilon,\xi}(\omega,\eta)$ that
    \begin{equation*}
    \begin{split}
    Q_{\varepsilon,\xi}(\omega,\eta)
    = & \langle \omega,\eta \rangle_\varepsilon
    - p\int_{\Omega}(P_{\varepsilon,\Omega}U_{\varepsilon,\xi})^{p-1}\omega\eta dx
    +\int_{\Omega}\Phi[P^2_{\varepsilon,\Omega}U_{\varepsilon,\xi}]\omega\eta  dx  \\
    & +2\int_{\Omega}\Phi[\omega P_{\varepsilon,\Omega}U_{\varepsilon,\xi}]
    P_{\varepsilon,\Omega}U_{\varepsilon,\xi}(x)\eta(x)  dx.
    \end{split}
    \end{equation*}
    Obviously, $\langle \omega,\eta \rangle_\varepsilon\leq \|\omega\|_{\varepsilon}\|\eta\|_{\varepsilon}$. On the one hand,
    \begin{equation*}
    \begin{split}
    \left|\int_{\Omega}(P_{\varepsilon,\Omega}U_{\varepsilon,\xi})^{p-1}\omega\eta dx \right|\leq C'\left[\int_{\Omega}\omega^2 dx\right]^{1/2}\left[\int_{\Omega}\eta^2 dx\right]^{1/2}\leq C''\|\omega\|_{\varepsilon}\|\eta\|_{\varepsilon},
    \end{split}
    \end{equation*}
    and from Lemma \ref{lemeguc}, we have
    \begin{equation*}
    \begin{split}
    \left|\int_{\Omega}\Phi[P^2_{\varepsilon,\Omega}U_{\varepsilon,\xi}]\omega\eta  dx \right|
    \leq & \int_{\Omega}\Phi[U^2_{\varepsilon,\xi}]|\omega||\eta|  dx
    \leq C\int_{\Omega}|\omega||\eta|dx  \\
    \leq & C'\left[\int_{\Omega}\omega^2 dx\right]^{1/2}\left[\int_{\Omega}\eta^2 dx\right]^{1/2}\leq C''\|\omega\|_{\varepsilon}\|\eta\|_{\varepsilon}.
    \end{split}
    \end{equation*}
    On the other hand, since $\Phi[\omega P_{\varepsilon,\Omega}U_{\varepsilon,\xi}]$ satisfies
    \begin{equation*}
    \begin{split}
    (-\Delta)^s\Phi[\omega P_{\varepsilon,\Omega}U_{\varepsilon,\xi}]=\omega P_{\varepsilon,\Omega}U_{\varepsilon,\xi}\quad \mbox{in}\quad \Omega,
    \quad \Phi[\omega P_{\varepsilon,\Omega}U_{\varepsilon,\xi}]=0\quad \mbox{in}\quad \mathbb{R}^N\backslash\Omega,
    \end{split}
    \end{equation*}
    then we have
    \begin{equation*}
    \begin{split}
    & \int_{\Omega}\Phi[\omega P_{\varepsilon,\Omega}U_{\varepsilon,\xi}](-\Delta)^s\Phi[\omega P_{\varepsilon,\Omega}U_{\varepsilon,\xi}]dx \\
    = & \int_{\Omega}\omega P_{\varepsilon,\Omega}U_{\varepsilon,\xi}\Phi[\omega P_{\varepsilon,\Omega}U_{\varepsilon,\xi}]dx \\
    \leq & \left(\int_{\Omega}(\omega P_{\varepsilon,\Omega}U_{\varepsilon,\xi})^2dx\right)^{1/2}
    \left(\int_{\Omega}(\Phi[\omega P_{\varepsilon,\Omega}U_{\varepsilon,\xi}])^2dx\right)^{1/2}  \\
    \leq & \left(\int_{\Omega}(\omega P_{\varepsilon,\Omega}U_{\varepsilon,\xi})^2dx\right)^{1/2}
    \left(\int_{\Omega}|(-\Delta)^{\frac{s}{2}}\Phi[\omega P_{\varepsilon,\Omega}U_{\varepsilon,\xi}]|^2dx\right)^{1/2},
    \end{split}
    \end{equation*}
    then
    \begin{equation}\label{eqwpwe}
    \begin{split}
    \int_{\Omega}|(-\Delta)^{\frac{s}{2}}\Phi[\omega P_{\varepsilon,\Omega}U_{\varepsilon,\xi}]|^2dx
    \leq \int_{\Omega}(\omega P_{\varepsilon,\Omega}U_{\varepsilon,\xi})^2dx.
    \end{split}
    \end{equation}
    By H\"{o}lder's inequality and (\ref{eqwpwe}), we obtain
    \begin{equation}\label{qbgx}
    \begin{split}
    \left|\int_{\Omega}\Phi[\omega P_{\varepsilon,\Omega}U_{\varepsilon,\xi}]
    P_{\varepsilon,\Omega}U_{\varepsilon,\xi}(x)\eta(x)  dx\right|
    \leq & \left[\int_{\Omega}(\Phi[\omega P_{\varepsilon,\Omega}U_{\varepsilon,\xi}])^2 dx\right]^{\frac{1}{2}}\left[\int_{\Omega}
    (\eta P_{\varepsilon,\Omega}U_{\varepsilon,\xi} )^2 dx\right]^{\frac{1}{2}}\\
    \leq & C'\left[\int_{\Omega}|(-\Delta)^{\frac{s}{2}}\Phi[\omega P_{\varepsilon,\Omega}U_{\varepsilon,\xi}]|^2 dx\right]^{\frac{1}{2}}\|\eta\|_{\varepsilon}\\
    \leq & C''\left[\int_{\Omega}(\omega P_{\varepsilon,\Omega}U_{\varepsilon,\xi})^2dx\right]^{\frac{1}{2}}\|\eta\|_{\varepsilon}  \\
    \leq & C'''\|\omega\|_{\varepsilon}\|\eta\|_{\varepsilon} .
    \end{split}
    \end{equation}
    Hence (\ref{qbw}) holds, and (\ref{qba}) is a direct result.
    \end{proof}

    We prove now the coercivity of the operator $\mathcal{A}_{\varepsilon,\xi}$ in the space $E_{\varepsilon,\xi}$. More precisely we have:

    \begin{lemma}\label{lemeqw}
    There are $\varepsilon_0>0$, $\tau>0$ and $R_0>0$, such that for $\varepsilon\in (0,\varepsilon_0]$, $\xi\in D_{\varepsilon,R}$ with $R\geq R_0$, we have
    \begin{equation}\label{eqwb}
    \|\mathcal{A}_{\varepsilon,\xi}\omega\|_{\varepsilon}\geq \tau\|\omega\|_{\varepsilon},\ \ \forall \omega\in E_{\varepsilon,\xi}.
    \end{equation}
    \end{lemma}

    \begin{proof}
    We follow the same idea as in \cite{dy1}. We argue by contradiction. Suppose that there are $\varepsilon_j\rightarrow 0$, $\xi^j\in \Omega$ with $\frac{d(\xi^j,\partial\Omega)}{\varepsilon_j}\rightarrow +\infty$, $\omega_j\in E_{\varepsilon_j,\xi^j}$ such that
    \begin{equation}\label{eqwaw}
    \|\omega_j\|_{\varepsilon_j}=\varepsilon^{\frac{N}{2}}_j\ \ \mbox{and}\ \ \|\mathcal{A}_{\varepsilon_j,\xi^j}\omega_j\|_{\varepsilon_j}=o(1)\|\omega_j\|_{\varepsilon_j}.
    \end{equation}
    We claim that for any $R>0$ fixed, $i=1,2,\ldots,N$,
    \begin{equation}\label{eqwwj}
    \int_{_{B_{\varepsilon_j R}(\xi^j)}}\omega^2_j dx=\varepsilon^{N}_j o(1).
    \end{equation}
    Let
    \begin{equation*}
    \widetilde{\omega}_j(x)=\omega_j(\varepsilon_j x+\xi^j),\ \ Z_{ij}(x)=\frac{\partial P_{\varepsilon_j,\Omega}U_{\varepsilon_j,\xi^j}}{\partial \xi^j_i}(\varepsilon_j x +\xi^j),
    \end{equation*}
    then it follows from (\ref{eqwaw}) that
    \begin{equation}\label{eqwwx}
    \begin{split}
    o(1)\|\psi\|_{\varepsilon_j}
    = & \int_{\Omega_{\varepsilon_j,\xi^j}}\left(\psi(-\Delta)^s \widetilde{\omega}_j+  \widetilde{\omega}_j \psi\right)
    -p\int_{\Omega_{\varepsilon_j,\xi^j}}(P_{\varepsilon_j,\Omega_{\varepsilon_j,\xi^j}}U)^{p-1}\widetilde{\omega}_j\psi\\
    & +\int_{\Omega_{\varepsilon_j,\xi^j}}\Phi[P^2_{\varepsilon_j,\Omega_{\varepsilon_j,\xi^j}}U]\widetilde{\omega}_j\psi +2\int_{\Omega_{\varepsilon_j,\xi^j}}\Phi[\widetilde{\omega}_j P_{\varepsilon_j,\Omega_{\varepsilon_j,\xi^j}}U]
    P_{\varepsilon_j,\Omega_{\varepsilon_j,\xi^j}}U\psi
    \end{split}
    \end{equation}
    for any
    \begin{equation*}
    \psi\in \widetilde{E}_j=\left\{\psi\in H^s_0(\Omega_{\varepsilon_j,\xi^j}): \int_{\Omega_{\varepsilon_j,\xi^j}}\left(\psi(-\Delta)^s Z_{ij}+  Z_{ij} \psi\right)=0,\ i=1,\ldots,N\right\},
    \end{equation*}
    where $P_{\varepsilon_j,\Omega_{\varepsilon_j,\xi^j}}U$ satisfies
    \begin{equation*}
    (-\Delta)^s P_{\varepsilon_j,\Omega_{\varepsilon_j,\xi^j}}U + P_{\varepsilon_j,\Omega_{\varepsilon_j,\xi^j}}U=U^p\quad \mbox{in}\quad \Omega_{\varepsilon_j,\xi^j}.
    \end{equation*}
    Since $\widetilde{\omega}_j$ is bounded in $H^s(\mathbb{R}^N)$, we may assume that there is an $\widetilde{\omega}\in H^s(\mathbb{R}^N)$ such that
    \begin{equation*}
    \widetilde{\omega}_j\rightharpoonup \widetilde{\omega},\ \ \mbox{weakly in}\ \ H^s(\mathbb{R}^N).
    \end{equation*}
    By Proposition \ref{prohlsi} the $(HLS)$ inequality, we obtain that for any $\omega, \eta\in \mathbb{H}$,
    \begin{equation}\label{yyhls}
    \begin{split}
    \int_{\Omega}\Phi[\omega P_{\varepsilon,\Omega}U_{\varepsilon,\xi}]P_{\varepsilon,\Omega}U_{\varepsilon,\xi}\eta  dx
    \leq & c_{N,s}\int_{\Omega}\int_{\Omega}\frac{\omega (y) P_{\varepsilon,\Omega}U_{\varepsilon,\xi}(y)}{|x-y|^{N-2s}}dyP_{\varepsilon,\Omega}U_{\varepsilon,\xi}(x)\eta(x)  dx   \\
    \leq & C\left(\int_{\Omega}(\omega P_{\varepsilon,\Omega}U_{\varepsilon,\xi})^{\frac{2N}{N+2s}}\right)^{\frac{N+2s}{2N}}
    \left(\int_{\Omega}(\eta P_{\varepsilon,\Omega}U_{\varepsilon,\xi})^{\frac{2N}{N+2s}}\right)^{\frac{N+2s}{2N}} \\
    \leq & C\left(\int_{\Omega}|\omega|^2\right)^{\frac{1}{2}}\left(\int_{\Omega}|\eta|^2\right)^{\frac{1}{2}}
    \left(\int_{\Omega}(P_{\varepsilon,\Omega}U_{\varepsilon,\xi})^{\frac{N}{s}}\right)^{\frac{2s}{N}} \\
    \leq & C\varepsilon^{2s}\|\omega\|_\varepsilon\|\eta\|_\varepsilon,
    \end{split}
    \end{equation}
    and by Lemma \ref{lemeguc}, it holds that
    \begin{equation*}
    \begin{split}
    \int_{\Omega}\Phi[P^2_{\varepsilon,\Omega}U_{\varepsilon,\xi}]\omega \eta  dx
    \leq & C\varepsilon^{2s}\|\omega\|_\varepsilon\|\eta\|_\varepsilon.
    \end{split}
    \end{equation*}

    Let $j\rightarrow\infty$ in (\ref{eqwwx}), then we obtain that $\widetilde{\omega}$ satisfies
    \begin{equation*}
    \int_{\mathbb{R}^N}\left((-\Delta)^{\frac{s}{2}} \widetilde{\omega}(-\Delta)^{\frac{s}{2}}\psi+  \widetilde{\omega} \psi\right)
    -p\int_{\mathbb{R}^N}U^{p-1}\widetilde{\omega}\psi=0,\ \ \mbox{for all}\ \ \psi\in C^{\infty}_0(\mathbb{R}^N).
    \end{equation*}
    Now, using (\ref{und}), it holds
    \begin{equation}\label{eqwwd}
    \widetilde{\omega}=\sum^N_{i=1}b_i\frac{\partial U}{\partial x_i},\ \ \mbox{for some}\ \ b_i\in\mathbb{R}.
    \end{equation}
    Recalling that for $\widetilde{\omega}\in E_{\varepsilon_j,\xi^j}$, we have $\widetilde{\omega}_j\in \widetilde{E}_j$, hence in (\ref{eqwwd}), we deduce that $\widetilde{\omega}=0$ and (\ref{eqwwj}) follows.
    Noting that $P_{\varepsilon_j,\Omega}U_{\varepsilon_j,\xi^j}$ is small in $\Omega\backslash B_{\varepsilon_j R}(\xi^j)$ if $R>0$ is large, and then from Lemma \ref{lemeguc} and (\ref{yyhls}), we get
    \begin{equation*}
    \begin{split}
    o_j(1)\varepsilon_j^N
    = & \langle \mathcal{A}_{\varepsilon_j,\xi^j}\omega_j,\omega_j \rangle_{\varepsilon_j} \\
    = & \|\omega_j\|^2_{\varepsilon_j}
    -\int_{\Omega}(P_{\varepsilon_j,\Omega}U_{\varepsilon_j,\xi^j})^{p-1}\omega^2_j dx \\
    & +\int_{\Omega}\Phi[P^2_{\varepsilon_j,\Omega}U_{\varepsilon_j,\xi^j}]\omega_j^2  dx
    +2\int_{\Omega}\Phi[\omega_j P_{\varepsilon_j,\Omega}U_{\varepsilon_j,\xi^j}]
    P_{\varepsilon_j,\Omega}U_{\varepsilon_j,\xi^j}\omega_j   dx\\
    \geq & \varepsilon_j^N
    - p\int_{\Omega}(P_{\varepsilon_j,\Omega}U_{\varepsilon_j,\xi^j})^{p-1}\omega^2_j dx
    +O(\varepsilon_j^2)\int_{\Omega}\omega_j^2  dx  +O(\varepsilon_j^2)\|\omega_j\|^2_{\varepsilon_j}\\
    = &  \varepsilon_j^N
    - p \int_{B_{\varepsilon_j R}(\xi^j)}(P_{\varepsilon_j,\Omega}U_{\varepsilon_j,\xi^j})^{p-1}\omega^2_j dx
    -p\int_{\Omega\backslash B_{\varepsilon_j R}(\xi^j)}(P_{\varepsilon_j,\Omega}U_{\varepsilon_j,\xi^j})^{p-1}\omega^2_j dx+o(\varepsilon_j^N)  \\
    = & (1+o_j(1)+o_R(1))\varepsilon_j^N  \\
    \geq & c_0 \varepsilon_j^N,
    \end{split}
    \end{equation*}
    for some small $c_0>0$, this is a contradiction. Thus the result holds.
    \end{proof}

    Furthermore, let us estimate the behavior of $R_\varepsilon(\omega)$.
    \begin{lemma}\label{lemrw}
    There exist a constant $C>0$ independent of $\varepsilon$, such that for any $\omega\in \mathbb{H}$ with $|\omega|\leq |P_{\varepsilon,\Omega}U_{\varepsilon,\xi}|$, we have
    \begin{equation}\label{rw0}
    \begin{split}
    \|R^{(i)}_{\varepsilon}(\omega)\|_\varepsilon\leq C(\varepsilon^{2s-\frac{N}{2}}\|\omega\|^{3-i}_\varepsilon+\varepsilon^{2s-N}\|\omega\|^{4-i}_\varepsilon
    +\varepsilon^{-\frac{N}{2}\min\{1,p-1\}}\|\omega\|^{\min\{3-i,p+1-i\}}_{\varepsilon}),
    \end{split}
    \end{equation}
    for $i=0,1,2.$
    \end{lemma}

    \begin{proof}
    This proof follows adapting the same arguments as in \cite[Lemma 3.10]{dsa}, see also \cite[Lemma 3.1]{lz17}. For each $\omega\in \mathbb{H}$ with $|\omega|\leq |P_{\varepsilon,\Omega}U_{\varepsilon,\xi}|$, recalling the definition $R_{\varepsilon}(\omega)$ in (\ref{rr}) that
    \begin{equation*}
    \begin{split}
    R_{\varepsilon}(\omega)
    = &\int_{\Omega}\frac{(P_{\varepsilon,\Omega}U_{\varepsilon,\xi})^{p+1}-(P_{\varepsilon,\Omega}U_{\varepsilon,\xi}+\omega)^{p+1}}{p+1}dx
    +\int_{\Omega}(P_{\varepsilon,\Omega}U_{\varepsilon,\xi})^p\omega dx
    +\frac{p}{2}\int_{\Omega}(P_{\varepsilon,\Omega}U_{\varepsilon,\xi})^{p-1}\omega^2 dx  \\
    & + \int_{\Omega}\Phi[\omega^2]P_{\varepsilon,\Omega}U_{\varepsilon,\xi}\omega dx
    +\frac{1}{4}\int_{\Omega}\Phi[\omega^2]\omega^2 dx  \\
    := & R_{1,\varepsilon}(\omega)+R_{2,\varepsilon}(\omega).
    \end{split}
    \end{equation*}
    Then we can obtain that
    \begin{equation*}
    \begin{split}
    \left|R_{1,\varepsilon}(\omega)\right|
    \leq C \int_{\Omega}\left|(P_{\varepsilon,\Omega}U_{\varepsilon,\xi})^{p-2}\omega^3\right| dx.
    \end{split}
    \end{equation*}
    One the one hand, if $1<p<2$, from (\ref{deflqfh}) we can get
    \begin{equation*}
    \begin{split}
    \left|R_{1,\varepsilon}(\omega)\right|
    \leq & C \int_{\Omega}\left|\frac{\omega}{P_{\varepsilon,\Omega}U_{\varepsilon,\xi}}\right|^{2-p}|\omega|^{1+p}dx
    \leq  C \int_{\Omega}|\omega|^{1+p}dx
    \leq  C \varepsilon^{(1-\frac{p}{2})N}\|\omega\|^{p+1}_{\varepsilon},
    \end{split}
    \end{equation*}
    and
    \begin{equation*}
    \begin{split}
    \left|R'_{1,\varepsilon}(\omega)\eta\right|
    \leq & C \int_{\Omega}|\omega|^{p}|\eta| dx
    \leq  C \left(\int_{\Omega}|\omega|^{p+1}dx\right)^{\frac{p}{p+1}}\left(\int_{\Omega}|\eta|^{p+1}dx\right)^{\frac{1}{p+1}}
    \leq  C \varepsilon^{(1-\frac{p}{2})N}\|\omega\|^p_{\varepsilon}\|\eta\|_{\varepsilon},
    \end{split}
    \end{equation*}
    moreover
    \begin{equation*}
    \begin{split}
    \left|R''_{1,\varepsilon}(\omega)(\eta_1,\eta_2)\right|
    \leq & C \int_{\Omega}|\omega|^{p-1}|\eta_1||\eta_2| dx  \\
    \leq & C \left(\int_{\Omega}|\omega|^{p+1}dx\right)^{\frac{p-1}{p+1}}\left(\int_{\Omega}|\eta_1|^{p+1}dx\right)^{\frac{1}{p+1}}
    \left(\int_{\Omega}|\eta_2|^{p+1}dx\right)^{\frac{1}{p+1}}  \\
    \leq & C \varepsilon^{(1-\frac{p}{2})N}\|\omega\|^{p-1}_{\varepsilon}\|\eta_1\|_{\varepsilon}\|\eta_2\|_{\varepsilon}.
    \end{split}
    \end{equation*}
    On the other hand, for the case $p>2$, we also can obtain that
    \begin{equation*}
    \begin{split}
    \left|R_{1,\varepsilon}(\omega)\right|
    \leq & C \varepsilon^{-\frac{N}{2}}\|\omega\|^{3}_{\varepsilon},  \\
    \left|R'_{1,\varepsilon}(\omega)\eta\right|
    \leq & C\varepsilon^{-\frac{N}{2}}\|\omega\|^2_{\varepsilon}\|\eta\|_{\varepsilon},  \\
    \left|R''_{1,\varepsilon}(\omega)(\eta_1,\eta_2)\right|
    \leq & C \varepsilon^{-\frac{N}{2}}\|\omega\|_{\varepsilon}\|\eta_1\|_{\varepsilon}\|\eta_2\|_{\varepsilon}.
    \end{split}
    \end{equation*}
    Then by using (HLS) inequality and (\ref{deflqfh}), (\ref{epuu}), we obtain that
    \begin{equation*}
    \begin{split}
    R_{2,\varepsilon}(\omega)
    = & \int_{\Omega}\Phi[\omega^2]P_{\varepsilon,\Omega}U_{\varepsilon,\xi}\omega dx
    +\frac{1}{4}\int_{\Omega}\Phi[\omega^2]\omega^2 dx  \\
    \leq & C\left(\int_{\Omega}|\omega|^{\frac{4N}{N+2s}}\right)^{\frac{N+2s}{2N}}
    \left(\int_{\Omega}(\omega P_{\varepsilon,\Omega}U_{\varepsilon,\xi})^{\frac{2N}{N+2s}}\right)^{\frac{N+2s}{2N}}
    + C\varepsilon^{2s-N}\|\omega\|^4_\varepsilon\\
    \leq & C\left(\int_{\Omega}|\omega|^{\frac{4N}{N+2s}}\right)^{\frac{N+2s}{2N}}
    \left(\int_{\Omega}|\omega|^2\right)^{\frac{1}{2}}
    \left(\int_{\Omega}|P_{\varepsilon,\Omega}U_{\varepsilon,\xi}|^{\frac{N}{s}}\right)^{\frac{s}{N}}
    + C\varepsilon^{2s-N}\|\omega\|^4_\varepsilon\\
    = & O(\varepsilon^{2s-\frac{N}{2}}\|\omega\|^3_\varepsilon+\varepsilon^{2s-N}\|\omega\|^4_\varepsilon).
    \end{split}
    \end{equation*}
    Similarly,
    \begin{equation*}
    \begin{split}
    \left|R'_{2,\varepsilon}(\omega)\eta\right|= & O(\varepsilon^{2s-\frac{N}{2}}\|\omega\|^2_\varepsilon\|\eta\|_\varepsilon+\varepsilon^{2s-N}\|\omega\|^3_\varepsilon\|\eta\|_\varepsilon),\\
    \left|R''_{2,\varepsilon}(\omega)\eta_1\eta_2\right|= & O(\varepsilon^{2s-\frac{N}{2}}\|\omega\|_\varepsilon\|\eta_1\|_\varepsilon\|\eta_2\|_\varepsilon
    +\varepsilon^{2s-N}\|\omega\|^2_\varepsilon\|\eta_1\|_\varepsilon\|\eta_2\|_\varepsilon).
    \end{split}
    \end{equation*}
    Consequently,
    \begin{equation*}
    \begin{split}
    \|R^{(i)}_{\varepsilon}(\omega)\|_\varepsilon
    \leq & \|R^{(i)}_{1,\varepsilon}(\omega)\|_\varepsilon+\|R^{(i)}_{2,\varepsilon}(\omega)\|_\varepsilon  \\
    \leq & C(\varepsilon^{2s-\frac{N}{2}}\|\omega\|^{3-i}_\varepsilon+\varepsilon^{2s-N}\|\omega\|^{4-i}_\varepsilon
    +\varepsilon^{-\frac{N}{2}\min\{1,p-1\}}\|\omega\|^{\min\{3-i,p+1-i\}}_{\varepsilon})
    \end{split}
    \end{equation*}
    for $i=0,1,2.$
    \end{proof}

    Next we prove:
    \begin{proposition}\label{profu}
    There is an $\varepsilon_0>0$ such that for each $\varepsilon\in (0,\varepsilon_0]$, there exists a unique $C^1$-map $\omega_{\varepsilon,\xi}: D_{\varepsilon,R}\rightarrow E_{\varepsilon,\xi}$ such that $J'_\varepsilon(P_{\varepsilon,\Omega}U_{\varepsilon,\xi}+\omega_{\varepsilon,\xi})\in E_{\varepsilon,\xi}^{\perp}$, that is
    \begin{equation}\label{fu0}
    \left\langle \frac{\partial J_\varepsilon\left(P_{\varepsilon,\Omega}U_{\varepsilon,\xi}+\omega_{\varepsilon,\xi}\right)}{\partial \omega},\eta\right\rangle_\varepsilon=0,\ \ \forall \eta\in E_{\varepsilon,\xi},
    \end{equation}
    where $J_\varepsilon$ is defined in {\rm (\ref{efPb})}. Moreover, we have
    \begin{equation}\label{fuo}
    \|\omega_{\varepsilon,\xi}\|_\varepsilon=\varepsilon^{\frac{N}{2}}O\left(\varepsilon^{\kappa}\right),\ \ \forall \xi\in D_{\varepsilon,R}.
    \end{equation}
    where $\kappa=\frac{2s+\zeta_0(N-2s)}{2}$ and $\zeta_0$ is given in Lemma \ref{lemel}.
    \end{proposition}

    \begin{proof}
    This proof is shown in \cite[Proposition 4.18]{adm}. As stated in Section \ref{sectrt}, let
    \begin{equation*}
    K(\xi,\omega)=J_\varepsilon\left(P_{\varepsilon,\Omega}U_{\varepsilon,\xi}+\omega\right),\ \xi\in D_{\varepsilon,R},\ \omega\in E_{\varepsilon,\xi}.
    \end{equation*}
    Expanding $K$ near $\omega=0$, it holds
    \begin{equation*}
    K(\xi,\omega)=J_\varepsilon\left(P_{\varepsilon,\Omega}U_{\varepsilon,\xi}\right)+L_{\varepsilon,\xi}(\omega)+\frac{1}{2}Q_{\varepsilon,\xi}(\omega,\omega)+R_\varepsilon(\omega),
    \end{equation*}
    Then, (\ref{fu0}) is equivalent to find a critical point of $K(\xi,\cdot)$ in $E_{\varepsilon,\xi}$. Using Lemma \ref{lemeqw}, we see that $\mathcal{A}_{\varepsilon,\xi}$ is invertible in $E_{\varepsilon,\xi}$ and there exists a constant $C$ independent of $\varepsilon$ and $\xi$, such that
    \begin{equation*}
    \|\mathcal{A}_{\varepsilon,\xi}\|_\varepsilon\leq  C.
    \end{equation*}
    Thus finding critical point for $K(\xi,\cdot)$ in $E_{\varepsilon,\xi}$ is equivalent to solving
    \begin{equation}\label{fulqr}
    l_{\varepsilon,\xi}+\mathcal{A}_{\varepsilon,\xi}\omega+R'_{\varepsilon}(\omega)=0.
    \end{equation}
    Rewrite (\ref{fulqr}) as $\omega=-\mathcal{A}_{\varepsilon,\xi}^{-1}(l_{\varepsilon,\xi}+R'_{\varepsilon}(\omega))$ and let
    \begin{equation*}
    G(\omega)=-\mathcal{A}_{\varepsilon,\xi}^{-1}(l_{\varepsilon,\xi}+R'_{\varepsilon}(\omega)),\ \ \xi\in D_{\varepsilon,R},\ \ \omega\in E_{\varepsilon,\xi}.
    \end{equation*}
    Hence, the problem reduces to finding a fixed point of the map $G$. For any $\omega_1,\ \omega_2\in E_{\varepsilon,\xi}$ with
    $\|\omega_1\|_{\varepsilon}, \|\omega_2\|_{\varepsilon}\leq \varepsilon^{\frac{N}{2}+\kappa}$  and $|\omega_1|, |\omega_2|\leq |P_{\varepsilon,\Omega}U_{\varepsilon,\xi}|$. Then by Lemma \ref{lemrw}, we get
    \begin{equation*}
    \begin{split}
    \|G(\omega_1)-G(\omega_2)\|_{\varepsilon}\leq & C \|R'_\varepsilon(\omega_1)-R'_\varepsilon(\omega_2)\|_{\varepsilon}  \\
    \leq & C_1 \|\int^1_0 R''_\varepsilon(t\omega_1+(1-t)\omega_2)(\omega_1-\omega_2) dt\|_{\varepsilon}  \\
    \leq & C_1 \max_{t\in [0,1]}\|R''_\varepsilon(t\omega_1+(1-t)\omega_2)\|_{\varepsilon}\|\omega_1-\omega_2\|_{\varepsilon} \\
    \leq & C_2 \left[\varepsilon^{2s+\kappa}+\varepsilon^{2s+2\kappa}+\varepsilon^{N\left(1-\frac{\min\{p+1,3\}}{2}\right)}\varepsilon^{(\frac{N}{2}+\kappa)\min\{p-1,1\}}\right]\|\omega_1-\omega_2\|_{\varepsilon} \\
    \leq & C_3 \varepsilon^{\kappa\min\{p-1,1\}}\|\omega_1-\omega_2\|_{\varepsilon}.
    \end{split}
    \end{equation*}
    Using the fact that $\kappa\min\{p-1,1\}>0$, we can choose $0<\nu<1$ such that
    \begin{equation}\label{fugw12}
    \begin{split}
    \|G(\omega_1)-G(\omega_2)\|_{\varepsilon}\leq & C_4 \varepsilon^{\nu}\|\omega_1-\omega_2\|_{\varepsilon},
    \end{split}
    \end{equation}
    so $G$ is a contraction map. On the other hand, for any $\omega\in E_{\varepsilon,\xi}$ with $\|\omega\|_{\varepsilon}\leq \varepsilon^{\frac{N}{2}+\kappa}$ and $|\omega|\leq |P_{\varepsilon,\Omega}U_{\varepsilon,\xi}|$, it holds that
    \begin{equation*}
    \begin{split}
    \|G(\omega)\|_{\varepsilon}\leq & C \|l_{\varepsilon,\xi}\|_\varepsilon+ C\|R_{\varepsilon}(\omega)\|_\varepsilon  \\
    \leq & C \|l_{\varepsilon,\xi}\|_\varepsilon+ C_5\left[\varepsilon^{2s-\frac{N}{2}}\|\omega\|^{3}_\varepsilon+\varepsilon^{2s-N}\|\omega\|^{4}_\varepsilon
    +\varepsilon^{N\left(1-\frac{\min\{p+1,3\}}{2}\right)}\|\omega\|_\varepsilon^{\min\{p,2\}}\right]  \\
    \leq & C \|l_{\varepsilon,\xi}\|_\varepsilon+ C_6\left[\varepsilon^{N+2s+3\kappa}+\varepsilon^{N+2s+4\kappa}+\varepsilon^{N\left(1-\frac{\min\{p+1,3\}}{2}\right)} \varepsilon^{\min\{p,2\}(\frac{N}{2}+\kappa)}\right]  \\
    = & C \|l_{\varepsilon,\xi}\|_\varepsilon+ C'_6\varepsilon^{\frac{N}{2}+\kappa\min\{p,2\}}.
    \end{split}
    \end{equation*}
    Hence, using Lemma \ref{lemel} and the fact $\min\{p,2\}>1$, we get
    \begin{equation}\label{fugw}
    \begin{split}
    \|G(\omega)\|_{\varepsilon}\leq  C_7\varepsilon^{\frac{N}{2}+\kappa}.
    \end{split}
    \end{equation}
    Combining (\ref{fugw12}) and (\ref{fugw}), we see that $G$ is a contraction map from $E_{\varepsilon,\xi}\cap B(0,\varepsilon^{\frac{N}{2}+\kappa})$ into itself. By the contraction mapping Theorem, we get the existence of unique $\omega_{\varepsilon,\xi}\in E_{\varepsilon,\xi}\cap B(0,\varepsilon^{\frac{N}{2}+\kappa})$ such that
    \begin{equation*}
    \omega_{\varepsilon,\xi}=G(\omega_{\varepsilon,\xi}).
    \end{equation*}
    Furthermore from (\ref{fugw}), we get
    \begin{equation*}
    \|\omega_{\varepsilon,\xi}\|_{\varepsilon}=\varepsilon^{\frac{N}{2}}O(\varepsilon^{\kappa}). 
    \end{equation*}
    We prove now that $\xi\rightarrow \omega_{\varepsilon,\xi}$ is $C^1$. In order to do that, let us denote
    \begin{equation*}
    \Psi(\varepsilon,\xi,\omega)=\omega-G(\omega),\ \ \omega\in E_{\varepsilon,\xi}.
    \end{equation*}
    Then, $\Psi(\varepsilon,\xi,\omega_{\varepsilon,\xi})=0$. On the other hand,
    \begin{equation*}
    \partial_{\omega}\Psi(\varepsilon,\xi,0)(\eta)=\eta+\mathcal{A}_{\varepsilon,\xi}^{-1}(R''_{\varepsilon}(\omega)\eta).
    \end{equation*}
    Using Lemma \ref{lemeqw} and Lemma \ref{lemrw}, we conclude that $\partial_{\omega}\Psi(\varepsilon,\xi,0)$ is, for $\varepsilon$ small, invertible as a small perturbation of the identity. Thus the result follows by the implicit function Theorem.
    \end{proof}

\section{{\bfseries Proof of main result}}\label{secteps}

\noindent{\bfseries Proof of Theorem \ref{thmcs}.}

    Considering $s\in (0,1)$, $2s<N\leq 6s$ 
    and $p\in (1,\frac{N+2s}{N-2s})$ for our problem (\ref{Pb}). Then we define
    \begin{equation*}
    D_{\varepsilon}:=\left\{\xi:\xi\in\Omega,\ d(\xi,\partial\Omega)\in \left[\varepsilon^{1-\varpi}, \varepsilon^{1-\mu}\right]\right\},
    \end{equation*}
    where $\varpi, \mu$ are fixed two positive constants satisfying

    \begin{equation}\label{ptdefvm}
    1>\mu>\frac{1}{3}>\varpi>\max\left\{\frac{N+4s}{6(N+2s)},\frac{N+4s}{3[(1+p)(N+2s)-N]}, \frac{N-2s}{3(N+2-2s)}, \frac{6-2N-14s}{3(N+2-2s)}\right\}.
    \end{equation}

    Consider $\inf_{\xi\in D_{\varepsilon}} M_\varepsilon(\xi)$, where
    \begin{equation*}
    M_\varepsilon(\xi)=J_\varepsilon(P_{\varepsilon,\Omega}U_{\varepsilon,\xi}+\omega_{\varepsilon,\xi}).
    \end{equation*}
    Let $\xi_{\varepsilon}\in D_{\varepsilon}$ be a minimum point of $M_\varepsilon(\xi)$ in $D_{\varepsilon}$. In order to prove that $P_{\varepsilon,\Omega}U_{\varepsilon,\xi_\varepsilon}+\omega_{\varepsilon,\xi_\varepsilon}$ is a solution of (\ref{Pb}), it's well known that as is shown in \cite[Lemma 7.16]{dddv}
    and also \cite[Lemma 4.19]{adm}, we only need to prove that $\xi_\varepsilon$ is an interior point of $D_{\varepsilon}$. It follows from Lemmas \ref{lemel}, \ref{lemqb}, \ref{lemrw}, and Propositions \ref{proejbu}, \ref{profu}, that
    \begin{equation}\label{pthmm}
    \begin{split}
    M_\varepsilon(\xi)= & J_\varepsilon(P_{\varepsilon,\Omega}U_{\varepsilon,\xi})
    +O\left(\|l_{\varepsilon,\xi}\|_{\varepsilon}\|\omega_{\varepsilon,\xi}\|_{\varepsilon}
    +\|\omega_{\varepsilon,\xi}\|_{\varepsilon}^2+R_\varepsilon(\omega_{\varepsilon,\xi})\right)  \\
    = & J_\varepsilon(P_{\varepsilon,\Omega}U_{\varepsilon,\xi})
    +O\left(\varepsilon^{N+2s+\zeta_0(N-2s)}\right)  \\
    = & \varepsilon^{N}A_1+\frac{1}{4}\varepsilon^{N+2s}A_2
    +\frac{1}{2}\tau_{\varepsilon,\xi}-\frac{1}{4}\varepsilon^{2N}B^2 H(\xi,\xi)+O\left(\varepsilon^{N+2s+\zeta_0(N-2s)}\right)  \\
    & +\varepsilon^{N}O\left(\left(\frac{\varepsilon}{d(\xi,\partial\Omega)}\right)^{(1+p)(N+2s)-N}
    +\mu_{\varepsilon,\xi}
    \right),
    \end{split}
    \end{equation}
    where $\zeta_0$ is given in Lemma \ref{lemel} and $\mu_{\varepsilon,\xi}$ is given in (\ref{defmu}).
    Then from (\ref{ptdefvm}), we obtain

    \begin{equation}\label{pthmm2}
    \begin{split}
    M_\varepsilon(\xi)
    = & \varepsilon^{N}A_1+\frac{1}{4}\varepsilon^{N+2s}A_2
    +\frac{1}{2}\tau_{\varepsilon,\xi}-\frac{1}{4}\varepsilon^{2N}B^2 H(\xi,\xi)+o\left(\varepsilon^{N+2s+\zeta(N-2s)}\right),
    \end{split}
    \end{equation}
    where $\zeta$ is a fixed constant with $\frac{1}{3}<\zeta<\min\{\zeta_0,\mu\}$ and closes to $\frac{1}{3}$.

    Let $\xi^*\in D_{\varepsilon}$ with $d(\xi^*,\partial\Omega)=\varepsilon^{1-\zeta}$.
    Noting that $H(\xi^*,\xi^*)\sim \frac{1}{d(\xi^*,\partial\Omega)^{N-2s}}$ as $d(\xi^*,\partial\Omega)\to 0$ (for similar result to Laplacian, see \cite[Proposition 6.7.1.]{cpy21}), then we obtain from Lemma \ref{lemetx} that
    \begin{equation}\label{prothmx}
    \begin{split}
    M_\varepsilon(\xi^*)\leq & \varepsilon^{N}A_1+\frac{1}{4}\varepsilon^{N+2s}A_2
    +C_1 \varepsilon^{N+\zeta(N+4s)} -C_2 \varepsilon^{N+2s+\zeta(N-2s)} +o\left(\varepsilon^{N+2s+\zeta(N-2s)}\right)  \\
    \leq & \varepsilon^{N}A_1+\frac{1}{4}\varepsilon^{N+2s}A_2
    -C_3 \varepsilon^{N+2s+\zeta(N-2s)}.
    \end{split}
    \end{equation}

    Next, for any $\xi'\in \Omega$ with $d(\xi',\partial\Omega)=\varepsilon^{1-\varpi}$, we get
    \begin{equation}\label{prothmdx}
    \begin{split}
    M_\varepsilon(\xi')\geq & \varepsilon^{N}A_1+\frac{1}{4}\varepsilon^{N+2s}A_2
    +C_1 \varepsilon^{N+\varpi(N+4s)}
    -C_2 \varepsilon^{N+2s+\varpi(N-2s)} +o\left(\varepsilon^{N+2s+\zeta(N-2s)}\right)  \\
    \geq & \varepsilon^{N}A_1+\frac{1}{4}\varepsilon^{N+2s}A_2
    +o\left(\varepsilon^{N+2s+\zeta(N-2s)}\right),
    \end{split}
    \end{equation}
    where $\varrho>0$ is small.
    Combining (\ref{prothmx}) and (\ref{prothmdx}), we see that if $\varepsilon>0$ is small enough, $M_\varepsilon(\xi)$ can not attain its minimum on $\{\xi: \xi\in \Omega,d(\xi,\partial\Omega)=\varepsilon^{1-\varpi}\}$.

    On the other hand, for any $\xi''\in \Omega$ with $d(\xi'',\partial\Omega)=\varepsilon^{1-\mu}$, we get
    \begin{equation}\label{prothmdd}
    \begin{split}
    M_\varepsilon(\xi'')\geq & \varepsilon^{N}A_1+\frac{1}{4}\varepsilon^{N+2s}A_2
    +C_1 \varepsilon^{N+\mu(N+4s)}
    -C_2 \varepsilon^{N+2s+\mu(N-2s)} +o\left(\varepsilon^{N+2s+\zeta(N-2s)}\right)  \\
    \geq & \varepsilon^{N}A_1+\frac{1}{4}\varepsilon^{N+2s}A_2
    +o\left(\varepsilon^{N+2s+\zeta(N-2s)}\right).
    \end{split}
    \end{equation}
    Combining (\ref{prothmx}) and (\ref{prothmdd}), we see that $\varepsilon>0$ is small enough,
    $M_\varepsilon(\xi)$ also can not attain its minimum on $\{\xi: \xi\in \Omega,d(\xi,\partial\Omega)=\varepsilon^{1-\mu}\}$.

    Thus we have proved that $\xi_\varepsilon$ is an interior point of $D_\varepsilon$. Since $\varpi<\frac{1}{3}$ can be chosen sufficiently close to $\frac{1}{3}$ and $1>\mu>\frac{1}{3}$ is arbitrary, we can obtain $d(\xi_{\varepsilon},\partial\Omega)\in (\varepsilon^{\frac{2}{3}+\theta}, \varepsilon^{\frac{2}{3}-\theta})$ for any small $\theta>0$ and the proof of Theorem \ref{thmcs} is complete.
    \qed

    \begin{remark}\label{remeb}{\rm
    In fact, taking $\varpi$ as in (\ref{ptdefvm}), then by a simply calculation, it holds that
    \begin{equation*}
    \begin{split}
    2N+2-(1-\varpi)(N-2s+2)> & \frac{4}{3}(N+s),  \\
    3N+4s-(1-\varpi)(N-2s+2)> & \frac{4}{3}(N+s),  \\
    N+\varpi\left[(1+p)(N+2s)-N\right]> & \frac{4}{3}(N+s).
    \end{split}
    \end{equation*}
    then we can choose $\frac{1}{3}<\zeta<\min\{\zeta_0,\mu\}$ closes to $\frac{1}{3}$ where $\zeta_0$ is given in Lemma \ref{lemel}, such that
    \begin{equation*}
    \begin{split}
    \varepsilon^{N}O\left(\mu_{\varepsilon,\xi}\right)= & o\left(\varepsilon^{N+2s+\zeta(N-2s)}\right), \\
    \varepsilon^{N}O\left(\left(\frac{\varepsilon}{d(\xi,\partial\Omega)}\right)^{(1+p)(N+2s)-N}\right)= & o\left(\varepsilon^{N+2s+\zeta(N-2s)}\right), \\
    \end{split}
    \end{equation*}
    then (\ref{pthmm2}) follows. Moreover,
    \begin{equation*}
    \frac{N+4s}{6(N+2s)}> \frac{N-2s}{3(N+2s)}\Leftrightarrow N< 8s,
    \end{equation*}
    and
    \begin{equation*}
    \frac{N+4s}{3[(1+p)(N+2s)-N]}> \frac{N+4s}{6[p(N+2s)-N/2]}\Leftrightarrow p>1,
    \end{equation*}
    thus (\ref{ptdefvm}) contains (\ref{defwj}) and (\ref{defw}) when $N<8s$ and $p>1$.
    }
    \end{remark}

    \begin{remark}\label{rem13s}{\rm
    We remark that it is essential to compare between $\tau_{\varepsilon,\xi}$ and $\varepsilon^{2N}B^2 H(\xi,\xi)$. Noticing the facts that $N+\frac{1}{3}(N+4s)=N+2s+\frac{1}{3}(N-2s)=\frac{4}{3}(N+s)$ and
    \begin{equation}\label{13sdx}
    \begin{split}
    N+\wp(N+4s)> & N+2s+\wp(N-2s),\ \ \mbox{for any}\ \wp>\frac{1}{3},  \\
    N+\wp(N+4s)< & N+2s+\wp(N-2s),\ \ \mbox{for any}\ \wp<\frac{1}{3},
    \end{split}
    \end{equation}
    thus (\ref{prothmx}), (\ref{prothmdx}) and (\ref{prothmdd}) hold.

    From (\ref{ptdefvm}),
    \begin{equation*}
     \begin{split}
     \frac{N+4s}{3[(1+p)(N+2s)-N]}<\frac{1}{3}\Leftrightarrow & p>1, \\
     \frac{6-2N-14s}{3(N+2-2s)}<\frac{1}{3}\Leftrightarrow & 3N+12s-4>0.
     \end{split}
    \end{equation*}
    Moreover, in Lemma \ref{lemel},
    \begin{equation*}
     \begin{split}
     \frac{N}{2}+2s>\frac{2N+2s}{3}\Leftrightarrow & N<8s.
     \end{split}
    \end{equation*}
    Furthermore, in Section \ref{sectrt}, we have stated that only for $N\leq 6s$, the energy functional associated to our problem is well defined. Combining above conditions with the initial condition $N>2s$, we can conclude that $N$ and $s$ are required satisfying 
    \begin{equation*}
    s\in (0,1)\quad\mbox{and}\quad 2s<N\leq 6,
    \end{equation*}
    i.e.,
    \begin{equation*}
     \begin{split}
     2\leq N\leq 5,\frac{N}{6}\leq s<1\quad \mbox{or}\quad N=1,\frac{1}{6}\leq s<\frac{1}{2}.
     \end{split}
    \end{equation*}
    }
    \end{remark}



    \end{document}